\documentclass[11pt]{amsart}
\usepackage[utf8]{inputenc}
\usepackage[usenames,dvipsnames,svgnames,table]{xcolor}	
\usepackage[colorlinks=true]{hyperref}
\hypersetup{colorlinks,
citecolor=brown, linkcolor=blue, linktocpage
}
\usepackage[english]{babel}
\usepackage[]{amsmath}
\usepackage{cases}
\usepackage{cite}
\usepackage{empheq}
\usepackage{cleveref} 
\usepackage{amssymb}
\usepackage{bbm}
\usepackage{enumerate}   
\usepackage{comment}
\usepackage{geometry}
\usepackage{amsfonts}
\usepackage{color}
\usepackage[]{amsthm} 
\usepackage[shortlabels]{enumitem}
\usepackage{underscore}
\usepackage{xspace}

\newtheorem{theorem}{Theorem}[section]
\newtheorem{lemm}[theorem]{Lemma}
\newtheorem{lemma}[theorem]{Lemma}
\newtheorem{coro}[theorem]{Corollary}
\newtheorem{prop}[theorem]{Proposition}

\newtheorem{defi}[theorem]{Definition}
\newtheorem{defn}[theorem]{Definition}
\newtheorem{nota}[theorem]{Notation}
\newtheorem{rema}[theorem]{Remark}

\def\E{\mathbb{E}}
\def\R{\mathbb{R}}
\def\P{\mathbb{P}}

\def\N{\mathbb{N}}

\def\F{\mathcal{F}}

\def\B{\mathbb{B}}
\def\1{\mathbbm{1}}
\def\Ee{\mathcal{E}}
\def\Ss{\mathcal{S}}
\def\Dd{\mathcal{D}}

\def\t{\tau_\eps}   
\def\g{h^\eps}   
\def\gi{h^\infty}   
\def\da{(-\Delta)^\alpha}   
\def\dea{\Delta_{\eps,\alpha}}   
\def\fgf{\text{FGF}_\alpha(D)}   

\def\eps{\varepsilon}
\def\Nn{\mathcal{N}}

\def\Gg{\mathcal{G}}
\def\Ff{\mathcal{F}}

\def\t{\tau_\eps}   
\def\g{h^\eps}   
\def\gi{h^\infty}   
\def\restriction#1#2{\mathchoice
              {\setbox1\hbox{${\displaystyle #1}_{\scriptstyle #2}$}
              \restrictionaux{#1}{#2}}
              {\setbox1\hbox{${\textstyle #1}_{\scriptstyle #2}$}
              \restrictionaux{#1}{#2}}
              {\setbox1\hbox{${\scriptstyle #1}_{\scriptscriptstyle #2}$}
              \restrictionaux{#1}{#2}}
              {\setbox1\hbox{${\scriptscriptstyle #1}_{\scriptscriptstyle #2}$}
              \restrictionaux{#1}{#2}}}
\def\restrictionaux#1#2{{#1\,\smash{\vrule height .8\ht1 depth .85\dp1}}_{\,#2}}

\newcommand{\corr}[1]{\color{black}{#1}\color{black}\xspace}
\title[Characterisations of GFF and FGF via dynamics]{A martingale-type characterisation of the Gaussian free field and fractional Gaussian free fields}
\author{Juhan Aru and Guillaume Woessner}
\date{}

\begin{document}

\begin{abstract}

We establish martingale-type characterisations for the continuum Gaussian free field (GFF) and for fractional Gaussian free fields (FGFs), using their connection to the stochastic heat equation and to fractional stochastic heat equations. 

The main theorem on the GFF generalizes previous results of similar flavour and the characterisation theorems on the FGFs are new. The proof strategy is to link resampling dynamics coming from a martingale-type decomposition to the stationary dynamics of the desired field, i.e. to the (fractional) stochastic heat equation.
\end{abstract}

\maketitle


\section{Introduction}

We prove martingale-type characterisation results for the continuum Gaussian free field (GFF) and for fractional Gaussian free fields (FGFs), using their connection to the stochastic heat equation and to fractional stochastic heat equations. 

Such characterisation theorems are well-known in the case of Brownian motion - e.g. the identification as the only continuous process with independent stationary increments or as the only continuous semimartingale with linear quadratic variation - and they help to explain the central role of Brownian motion in probability theory. One could ask if similar characterisations for the $d$-dimensional continuum Gaussian free field and its fractional counterparts might also bring some insight into its universality - for example the GFF is both the proven and the conjectured scaling limit of several natural discrete height functions \cite{NS, Kenyon_GFF, RiderVirag,BLR}. Characterisation theorems also help to understand better which properties are central to the underlying object.

\subsection{The case of the Gaussian free field}
The Dirichlet boundary condition continuum Gaussian Free Field (abbreviated GFF), which can be viewed as a generalization to higher dimensions of the Brownian Motion, is defined as follows. 
\begin{defi}
\label{GFFdef}
The GFF on a bounded regular domain $D$ is the centred Gaussian process $(h^D,f)_{f\in C_c^\infty(D)}$ whose covariance is given by
\begin{align*}
\E[(h^D,f)(h^D,g)]=\int_D\int_D f(x)G_D(x,y)g(y)dxdy.
\end{align*}
\end{defi}
Here, a regular domain is any connected open subset of $\R^d$ with finitely many boundary components, such that for every $x\in\partial D$ the frontier of a domain $D$, the Brownian motion $B_t$ starting from $x$ satisfies $\inf\{t>0 ~:~B_t \in D^c\}=0$ almost surely; this guarantees that the Green's function is well defined and positive definite. Further, boundedness guarantees that the Green's integral operator is compact. See \cite{BP, PW} for more detailed yet friendly introductions.One can also show that there exists a modification of the GFF which is almost surely in $\Ss_D$\footnote{$\Ss_D$ is the topological space of all distributions of compact support in $D$}. \corr{Further, thoughout the paper the zero-boundary Green function \(G_D\)
is normalized so that
\[
G_D(x,y)-\Phi_d(|x-y|)
\]
is bounded near \(x\), where
\[
\Phi_d(r)=
\begin{cases}
\dfrac{1}{2\pi}\log(1/r), & d=2,\\[1.2ex]
\dfrac{1}{(d-2)A_d}r^{2-d}, & d\ge3,
\end{cases}
\]
and \(A_d=|\mathbb S^{d-1}|\).}

In \cite{BPR18} the authors showed that for $d=2$, the GFF is the only conformally invariant field satisfying a domain Markov property under some continuity and moments assumptions; these moment assumptions were then lowered in \cite{BPR20}. Further, in \cite{AruPowell}, the authors proved a characterisation of the GFF in $d\geq 2$, which heuristically corresponds to the stationary independent increments property of Brownian motion. The proof there is simple, but is tailored to balls, or other domains with symmetry.

The first contribution of this paper is to prove a martingale-type characterisation for the continuum GFF, which also works in a more general geometric setting than the earlier works \cite{BPR18, BPR20, AruPowell}, e.g. in multiply connected domains. We prove that essentially the GFF is characterized by the following martingale-type decomposition property - in any ball the field can be written as a sum of a harmonic function and a zero mean, zero-boundary field with certain scaling. This is considerably weaker than the domain Markov property assumed in \cite{BPR18,AruPowell}, see the Remark \ref{remarkaconstant} below. A more precise statement of our decomposition goes as follows: 
\begin{defi} [Martingale-type decomposition]
\label{defiMTD}
We say that a random distribution $h^D\in \Ss_D$ admits the \emph{martingale type decomposition} (abbreviated MTD) if for any ball $B:=B(x,\eps) \subset D$ we have the decomposition
$$ h^D = \tilde{h}^B + \varphi_D^B,$$
where $\varphi_D^B$ belongs to $\Ss_D$ and corresponds to integrating with respect to a harmonic function when restricted to $B$ ; and $\tilde{h}^B$ belongs to $\Ss_B$ and satisfies the following conditions, where $\restriction{T}{D'}$ is the restriction of a distribution $T$ to the subdomain $D'$.
\begin{enumerate}
\item The martingale property: for any $f \in C^\infty_c(D)$
$$\E\left[(\tilde{h}^B,f) ~|~\sigma\big(\restriction{(h^D)}{D\setminus B}\big)\right] = 0 .$$
\item A condition on the second moments of the increments. First, the squared increment is deterministic, i.e. for any $f \in C^\infty_c(D)$
$$ \E\left[(\tilde{h}^B,f)^2|~\restriction{(h^D)}{D\setminus B}\right] = \E\left[(\tilde{h}^B,f)^2\right]. $$
Second, we have weak scaling: for some $C > 0$ and some positive bounded measurable function $a: D \to (0, \infty)$ 
\begin{align*}
& \E\left[(\tilde{h}^B,f)^2~|~\restriction{(h^D)}{D\setminus B}\right] \leq C \sup_{z\in D}|f(z)|^2\eps^{2+d};
& \frac{\E\left[(\tilde{h}^B,f)^2~|~\restriction{(h^D)}{D\setminus B}\right]}{\eps^{2+d}} \to_{\eps \to 0} a^2(x)f^2(x)
\end{align*}
where the convergence is uniform on compacts\footnote{One should note that $B$ depends on $\epsilon$.}.
\item The field $\tilde{h}^B$ satisfies zero boundary conditions in the sense defined just below. 
\end{enumerate}
\end{defi}

We will in fact have to add a weak condition on the fourth moment of $\tilde{h}^B$ too, though possibly this can be relaxed further. Also, both in this definition and in the theorem below we have to define what it means for a distribution to have zero boundary conditions. We follow here \cite{AruPowell}:

\begin{defi}[Zero boundary conditions]
\label{defiZB}
We say that a random distribution $h^D\in \Ss_D$ satisfies zero boundary conditions in the $L^2$ sense if the following holds. 

Consider a sequence $(g_n)_{n\geq 1}$ of positive and of uniformly bounded mass functions of $C^\infty_c(D)$ such that $\sup_{n\geq 1}\sup_{r>0}\sup_{x,y\in D_r} |g_n(x)/g_n(y)| < \infty$ where $D_r := \{z\in D\,:\, d(z,\partial D)=r\}$, and for every compact $K\subset D$ it holds that $K\cap \emph{supp}(g_n) = \emptyset$ for $n$ large enough. Then we have that $\E[(h^D,g_n)^2] \mapsto 0$ as $n\mapsto\infty$.
\end{defi}

\noindent The characterisation theorem can be now stated:
\begin{theorem}
\label{mainthm}
Let $d\geq 2$ and $D$ be a regular domain of $\R^d$. Suppose that a random distribution $h^D\in \Ss_D$ satisfies the following properties
\begin{itemize}
\item[A.] \underline{\emph{Martingale-type decomposition}} : as in Definition \ref{defiMTD}, and satisfying in addition for some $C_x > 0$ that is bounded on compacts
$$ \E\left[(\tilde{h}^{B(x,\eps)},f)^4\right] \leq C_x^4\sup_{z\in D}|f(z)|^4\eps^{2(2+d)}.$$
\item[B.] \underline{\emph{Uniformly bounded 4th moments}} : there is some $C_0 = C(D) > 0$ such that for any $f\in C_c^\infty(D)$, we have 
$$\E[(h^D,f)^4] \leq C_0^2 \sup_{z\in D}|f(z)|^4.$$
\item[C.] \underline{\emph{Zero boundary conditions}} as in the Definition \ref{defiZB}.
\end{itemize}
Then there exists $c\geq 0$ such that $c\,h^D$ is a GFF. 
\end{theorem}

\begin{rema}
\label{remarkaconstant}
Compared to \cite{AruPowell} (and \cite{BPR18}) the main difference is the weakening of the domain Markov property. Indeed, in these articles the field was decomposed in every ball as a sum of a harmonic function and an independent field that was a rescaled version of the original field. Here we do not demand independence of the decomposition nor ask much from the law of the non-harmonic part, other than being centred, with zero boundary and satisfying some scaling. The latter is necessary, as otherwise white noise would satisfy the conditions too. 

We also allow for a certain possible non-homogeneity with a non-constant $a(x)$ in the MTD; yet one should emphasise that our theorem shows that for all the fields that do satisfy MTD with our additional conditions, the function $a$ ends up being constant. 

Further notice that the 4th moment condition in A and B is asking more than \cite{AruPowell, BPR20} in terms of moments: indeed, the uniformity we ask here was not part of their assumptions. One could get away with less, e.g. by assuming a uniform bound only for certain symmetric mollifiers; this would already very simply follow from the stronger Markov decomposition in \cite{AruPowell, BPR20} combined with just the finiteness of fourth moments.

Finally, a simple sanity check is that the GFF itself does satisfy the conditions of the theorem with constant $a = \sqrt{\kappa_d}:=\sqrt{\frac{\text{Vol}(B(0,1))}{d(d+2)}}$.
\end{rema}
\noindent The proof strategy is as follows: 

\begin{itemize}
\item the martingale-type decomposition gives us a natural way to resample the field inside balls and using a Poisson point process of these resamplings on randomly located balls we construct a stochastic dynamics on random fields, whose marginal law is equal to the initial field for every time $t > 0$ (Definition \ref{defi_dyn});
\item we show that this dynamics when we resample in smaller and smaller balls and with higher and higher frequency converges to the stochastic heat equation (Theorem \ref{thm_dyn});
\item but now the unique stationary solution of the stochastic heat equation is the Gaussian free field and as at any $t > 0$ all of the approximated dynamics have the law of the original field, we conclude that the original field has to be the GFF.
\end{itemize}

In the case where the function $a(x)$ is constant, i.e. we assume some homogeneity in space, we can follow this sketch rather precisely and interestingly - in contrast to \cite{AruPowell, BPR18} we don't need an extra step to identify the covariance kernel. In the case $a(x)$ non-constant, we still need to make use of a separate step to identify the covariance as the Green's kernel, but that is standard by now.

This strategy for proving the characterisation theorem is quite robust. Here we exemplify the generality by using the same strategy to prove a characterisation theorem for the fractional Gaussian fields, e.g. like the fractional Brownian motion, and answering Question 6.5 of \cite{BPR18} for a range of parameters. It would also work for Riemann surfaces with boundary, \corr{partly answering} Question 6.6 of \cite{BPR18}. \corr{More precisely, on a compact Riemannian surface with boundary one would replace Brownian motion, the Green function, harmonic measure and the Laplacian by their intrinsic counterparts. The small-ball part of the argument is local: in normal coordinates the metric is Euclidean up to lower-order errors, and the normalization of the resampling rate absorbs the resulting dimensional constant. The rest of the proof uses only the martingale-type decomposition, the Green representation and the martingale problem for the heat equation, all of which have standard analogues in this setting.}

\subsection{The case of fractional Gaussian fields}

Fractional Gaussian free fields $\text{FGF}_\alpha(D)$ can be defined as follows \cite{fractional_laplacian_review}:
\begin{defi}
\label{FGFdef}
For $\alpha > 0$, the $\text{FGF}_\alpha(D)$ on a bounded domain with smooth boundary is the centred Gaussian process $(h^D_\alpha,f)_{f\in C^\infty_c(D)}$ whose covariance is
\begin{align*}
\E[(h^D_\alpha,f)(h^D_\alpha,g)]=\int_D\int_D f(x)G_D^\alpha(x,y)g(y)dxdy,
\end{align*}
where $G_D^\alpha$ is the fractional Green function associated to the fractional laplacian $(-\Delta)^\alpha$ on $D$.
\end{defi}
The operator $(-\Delta)^\alpha u(z)$ here is the Riesz fractional Laplacian given by
\begin{align}
    \label{laplacianfractionalintegral}
    (-\Delta)^\alpha u(z) = \lim_{\eps\mapsto 0} \int_{\R^d-B(z,\eps)} \frac{u(z)-u(y)}{|z-y|^{d+2\alpha}}dy,
\end{align}
where we extended the values of $u$ outside of $D$ by $0$ (see e.g. \cite{FGF_review, bogdan2009}). It is known that the fractional Green function $G_D^\alpha$ exists; in the case $D = \B$ where $\B$ is the unit ball, it can be explicitly computed, see Section \ref{section_frac_intro}. The case $\alpha = 1$ corresponds to the Gaussian free field, in the case $d = 1$ and $\alpha \in (1/2, 3/2)$ we obtain the fractional Brownian bridges with the Hurst parameter $H = \alpha - d/2$, see \cite{fractional_laplacian_review}. We will concentrate on $d =2$ and the range $\alpha \in (0,1)$ where the Riesz laplacian is the generator of a $2\alpha-$stable L\'evy process. \footnote{Notice that our parameter $\alpha$ is exactly the parameter $s$ in \cite{FGF_review}, but we chose this notation to be more coherent with the analysis literature (although there is a factor of two difference with for example \cite{fractional_laplacian_review}).}
We believe the same strategy also applies in dimension one at least in the regime $1/2<\alpha<1$. However, as this part of the paper is foremost a showcase that the strategy does generalise, we leave further extensions out - the 1d case would require separate notations and some separate technical estimates, e.g. in Lemma \ref{lemme_control_c}, yet because of exact expressions they should easily follow.

In the fractional case the characterisation theorems have not been explored too much; even in the case of fractional Brownian motion characterisation theorems are relatively new - \cite{mishura, hu2009}. \corr{Still, there has been renewed interest in processes related to the fractional Laplacian for $d\geq 2$ - e.g. \cite{FGF_review, cao2024, gunaratnam2024} but also \cite{garban2024} - and thus we hope that our statements and proofs would contribute to our general understanding of what does and what does not generalise when passing from the standard Laplacian to the fractional versions.}
\\

In essence the characterisation of the GFF obtained in Theorem \ref{mainthm} generalizes nicely to fractional Gaussian free fields $\text{FGF}_\alpha(D)$ in the regime $0 < \alpha < 1$. Indeed, we prove that the fractional Gaussian free field is characterised by having a martingale-type decomposition on each ball: the field inside is given by a $\alpha-$harmonic function plus a zero mean, zero boundary field as before. The reason for restricting to $\alpha \leq 1$ (we have already dealt with $\alpha = 1$ as it is the GFF case) is the following. When $\alpha > 1$,  the simple mean value property for $\alpha$-harmonic functions fails. Even though one can define some version of the Poisson kernel \cite{abatangelo2018integral}, it becomes too singular to extend arbitrary boundary values and in fact extra terms are needed: see the discussion below Theorem 1.1 and Theorem 1.6 in  \cite{abatangelo2018integral}.

\begin{defi}
\label{defiFMTD}
Let $d\geq 2$ and $\alpha \in (0,1)$. We say that a random distribution $h^D\in \Ss_D$ admits the \emph{(fractional) martingale type decomposition} (abbreviated FMTD) if for any ball $B:=B(x,\eps) \subset D$ we have the decomposition
$$ h^D = \tilde{h}^B + \varphi_D^B,$$
where $\varphi_D^B$ belongs to $\Ss_D$, corresponds to integrating with respect to an $\alpha$-harmonic function (ie $(-\Delta)^\alpha \varphi_D^B =0$) when restricted to $B$ 
and $\tilde{h}^B$ belongs to $\Ss_B$ and satisfies the following conditions. 

\begin{enumerate}
\item The martingale property: for any $f \in C^\infty_c(D)$
$$\E\left[(\tilde{h}^B,f) ~|~\sigma\big(\restriction{(h^D)}{D\setminus B}\big)\right] = 0 .$$
\item A condition on the second moments of the increments: the squared increment is deterministic, i.e. for any $f \in C^\infty_c(D)$
$$ \E\left[(\tilde{h}^B,f)^2|~\restriction{(h^D)}{D\setminus B}\right] = \E\left[(\tilde{h}^B,f)^2\right] $$
and we have weak scaling: for some $C > 0$ 
\begin{align*}
& \E\left[(\tilde{h}^B,f)^2~|~\restriction{(h^D)}{D\setminus B}\right] \leq C \sup_{z\in D}|f(z)|^2\eps^{2\alpha+d};
& \frac{\E\left[(\tilde{h}^B,f)^2~|~\restriction{(h^D)}{D\setminus B}\right]}{\eps^{2\alpha+d}} \to_{\eps \to 0} f^2(x)
\end{align*}
where the convergence is uniform on compacts.
\item The field $\tilde{h}^B$ satisfies zero boundary conditions in the sense of Lemma \ref{defiZB}. 
\end{enumerate}
\end{defi}
\begin{rema}Notice that we have given up here the potential local inhomogeneity allowed by $a: D \to [0,\infty)$ and normalized it to be $1$. In the case of the GFF - as we saw in Remark \ref{remarkaconstant} -  there were in the end no fields with $a(x)$ non-constant, and it does simplify our life as we can skip the step of determining the covariance, which would require a bit more potential theory of the fractional laplacian.
\end{rema}

There is one key difference from the GFF case stemming from the fact that in general the fractional Laplacian is non-local. Namely, for $\alpha \in (0,1)$ there is no localised mean-value property for $\alpha$-harmonic functions, i.e. for an $\alpha-$harmonic function $f$ we cannot express $f(z)$ by averaging $f$ against some smooth function in any neighbourhood of $z$. Indeed, the $\alpha-$mean kernel is both non-local and non-smooth, as can be seen from the exact expression in Section \ref{section_frac_intro}. As our fields are a priori distribution valued in the range of interest, this causes some annoyances and one would need to ask for more regularity - basically that integrating the field against the mean-value kernel is well defined. As adding just this as a condition didn't feel very natural, we opted to further strengthen the martingale-type property as well - we assume that the $\alpha-$harmonic function in any ball is given by the $\alpha-$Poisson extension of the field outside the ball (in the GFF case this is obtained as a consequence).

\begin{theorem}
\label{mainthmfrac}
Let $d\geq 2$ and $\alpha \in (0,1)$. Suppose that a random distribution $h^D\in \Ss_D$ has the following property
\begin{itemize}
\item[A'.] \underline{\emph{FMTD}} : as in Definition \ref{defiFMTD}, and satisfying in addition for some $C > 0$
$$ \E\left[(\tilde{h}^B,f)^4\right] \leq C^4\sup_{z\in B}|f(z)|^4\eps^{2(2\alpha+d)}.$$
We further assume that for all balls $B \subset D$ we can evaluate $h^D$ against the fractional Poisson kernel $P_B^\alpha(z,\cdot)$ \footnote{See \eqref{def_poisson_int_fractional} for a definition of $P_B^\alpha$. Here to evaluate  means that for any sequence of mollifiers $(\alpha_n)_n$ the limit $\lim_n h^D(P_B^\alpha(z,\cdot)\ast \alpha_n)$ exists and does not depend on the choice of the $\alpha_n$.}, and moreover we have $$(h^D,P^\alpha_{B}(z,\cdot))=\varphi_D^B(z).$$
\item[B'.] \underline{\emph{Uniformly bounded 4th moments}} : there exists a constant $C_0>0$ such that for any $f\in C^\infty_c(D)$, we have 
$$\E[(h^D,f)^4] \leq C_0^2 \sup_{z\in D}|f(z)|^4.$$
\item[C'.] \underline{\emph{Zero boundary conditions}} as in the Definition \ref{defiZB}.
\end{itemize}
Then there exists $c\geq 0$ such that $c\,h^D$ is a $\text{FGF}_\alpha(D)$. 
\end{theorem}

\begin{rema}
Notice that as $h^B$ is supported in $B$, the final part in condition A' can also be written as $(\varphi_D^B, P^\alpha_{B}(z, \cdot)) = \varphi_D^B(z)$.

The fact that the fractional Gaussian free field itself satisfies these conditions can be again checked explicitly, for the zero boundary conditions one can use the estimates in \cite[Thm~1.1]{ChenSong}.
\end{rema}

As mentioned earlier, the proof of the above theorem stays really pretty much the same. There are only some minor simplifications due to the restated martingale decomposition, and some modifications when proving approximations of $(-\Delta)^\alpha$, both detailed in the section about the fractional case. Due to the assumption that the local in-homogeneity $a(x)$ is constant, we can skip the step of determining the covariance altogether.
\subsection*{Acknowledgements}

The authors would like to thank an anonymous referee for a very careful reading and useful suggestions. We also thank Avi Mayorcas and Ellen Powell for helpful comments on an earlier version of the paper.
 Both authors are supported by Eccellenza grant 194648 of the Swiss National Science Foundation and are members of NCCR Swissmap.
 
\section{The case of the GFF}

In this section we prove the martingale-type characterisation for the GFF, i.e. Theorem \ref{mainthm}. In Section \ref{sec:prelimGFF} we start with preliminaries on the stochastic heat equation and then deal with some simple properties of the field $h^D$, culminating in the determination of the covariance kernel needed in the case where $a(x)$ is not constant. The second half of this section follows closely \cite{AruPowell} and is kept short.

Thereafter in Section \ref{section_dyn_def} we define the approximate resampling dynamics and state the convergence theorem for this dynamics; we conclude Theorem \ref{mainthm} modulo this convergence. Next, in Section \ref{section_laplacien} we identify an approximation of the Laplacian related to our resampling property. In Section \ref{section_epsilon} we identify the relevant martingales and argue that they converge.

Finally in Section \ref{section_conv} we obtain the convergence of the original dynamics, concluding the whole proof in the case of the GFF.

\subsection{Preliminaries}
\label{sec:prelimGFF}
\subsubsection{Preliminaries on the stochastic heat equation}

We gather here the definition and some main properties of the additive stochastic heat equation, see e.g. \cite{Walsh} for more details. We formulate it in the mild sense.
\corr{
\begin{defi}[Mild formulation of SHE]
Let $d\geq 2$, $D$ a bounded regular domain of $\R^d$, and $(\Omega,\Ff,(\Ff_t)_{t\geq 0},\mathbb{P})$ a filtered probability space. The $a$-SHE on $D$ is the SPDE, formally satisfying
\begin{align}
\label{SHE}
\begin{cases}
& \partial_t u - \Delta u = a\,\dot{W} \\
&  u(0,x) = u_0(x), \qquad x\in D\\
& u(t,x) = 0, \qquad x\in \partial D,~ t\geq 0
\end{cases}
\end{align} 
where $\dot{W}$ is a space-time white noise, $a(x)$ is a measurable bounded positive function such that $a\dot{W}$ is a centered Gaussian process defined on Borel subsets of $\R^+\times \R^d$ whose covariance is $\text{Cov}(a\dot{W}(A),a\dot{W}(B))=\int_{A\cap B} a^2(x)dx$. Further, $u_0 \in \Ff_0$ is the initial condition.

A mild solution to \eqref{SHE} is a stochastic process $(u_t)_{t\geq 0}$ on $(\Omega,\Ff)$ 
taking values in the space of distributions $\Ss_D$ and such that almost surely for every $t\geq 0$ and $f \in C^\infty_c(D)$
\begin{align}
\label{SHEsol}
(u_t,f) := \int_D u_0(x)\left(\int_D H_D(t,x,y)f(y)dy\right)dx + \int_0^t\int_D \left(\int_D H_D(t-s,x,y)a(y)f(y)dy\right)W(dx,ds),
\end{align}
where $H_D$ is the fundamental solution of \eqref{SHE} (also called Green function of the heat kernel) on $D$ and can be expressed by 
\begin{align}
\label{SHEgreen}
H_D(t,x,y) = \sum_{n\in\N} e^{-\mu_n t}f_n(x)f_n(y),
\end{align}
where we recall $(f_n)_{n \geq 1}$ are an ON basis of $L^2(D)$ given by the smooth zero boundary eigenfunctions of the Laplacian, with related eigenvalues $-\Delta f_n = \mu_n f_n$; the existence of the eigenelements is guaranteed by the fact that the domain is regular, see e.g. \cite{MP_BM}.
\end{defi}
}
Let us make some remarks. Here in the first term of the sum we recognize the solution to the deterministic homogeneous PDE associated to \eqref{SHE}, i.e. the heat equation, and the second term is the solution to \eqref{SHE} with $u_0=0$.

\corr{Also, in this mild formulation the zero boundary conditions are fixed by the kernel $H_D$. It is equivalent to having the boundary conditions as in Definition \ref{defiZB} as it comes down to identifying the Green's function, like in Proposition \ref{prop:Green}.}

As the stochastic heat equation is a linear parabolic SPDE, the theory of the long-time behaviour of its solutions is well-understood and in particular we have the following proposition, see also e.g. \cite{LototskyShah}
\begin{prop}[Existence and uniqueness of stationary solutions]~\\
\label{propGFFSHE}
Suppose that $a(x)$ is measurable, bounded and positive. Then there is a unique stationary mild solution to the $a-$SHE, i.e. a unique mild solution $(u_t)_{t \geq 0}$ with $u_t$ having the same law for all $t \geq 0$ and satisfying zero boundary conditions.

Further, this law is Gaussian and its covariance kernel can be computed directly. When $a(x)$ is constant and equal to $\sqrt{2}$ it satisfies $$\E(u_t(x)u_s(y)) = \sum_{n \geq 1} \frac{1}{\mu_n}e^{-\mu_n|s-t|}f_n(x)f_n(y),$$
which in particular reproduces the covariance kernel of the GFF for $s = t$.
\end{prop}

One way to characterise SPDEs is via the related martingale problem. In the concrete case of the SHE one can even be more precise. See Appendix \ref{appendix_walsh} for the proof of the following classical result.

\begin{prop}
\label{propSHEOU}
The unique solution to \eqref{SHE} as defined by \eqref{SHEsol} can be described as an element of $\Ss_D$ using the infinite sum
$$u_t := \sum_{n=0}^\infty A_n(t) f_n,$$
where the random variables $A_n(t)$ are Ornstein-Uhlenbeck processes with parameter $\mu_n$, \textit{ie} solutions to the Itô SDEs
\begin{align}
\label{propSHEOUeq}
dX_t = -\mu_n X_td_t + dB^n_t,
\end{align}
and the $B_t^n$ are Brownian motions of covariation given by $\langle B_t^n,B_t^m\rangle = t\langle af_n~|~a f_m\rangle$ where in the last term $\langle \cdot~|~\cdot\rangle$ is the usual scalar product. Also, the initial conditions take the form $A_n(0)=\langle u_0~|~f_n\rangle $.
\end{prop}

In fact, using the unique stationary law of Ornstein-Uhlenbeck processes, one can also deduce Proposition \ref{propGFFSHE} from Proposition \ref{propSHEOU}. We will use the following corollary: 

\begin{coro}
\label{cor:char}
Let $(h_t)_{t \geq 1}$ be a distribution valued process. 
Suppose for all $f \in C_c^\infty(D)$ we have that 
$$(M_t,f) := (h_t,f) - \int_0^t (\Delta h_s,f) ds$$
is a continuous centered Gaussian process with variance $\E[(M_t,f)^2]=t\langle af~|~a f\rangle$ and where $(\Delta h,f) = (h,\Delta f)$ for any distribution $h$. Then $h_t$ is the unique solution of the $a-$SHE starting from the zero boundary conditions. 
\end{coro}

\subsubsection{Some simple properties of the field $h^D$}

In this subsection, we state immediate consequences coming from our assumptions on the field $h^D$ in Theorem \ref{mainthm}, culminating with a definition and an estimate on the 2-point function. This subsection follows closely \cite{AruPowell} Section 2.

\begin{lemm}\label{lemorth}
The following orthogonality properties are immediate from the martingale-type decomposition: 
$$\E[\tilde{h}^B ~|~\varphi_D^B]=0 \qquad \text{and} \qquad \E[\tilde{h}^B\,\varphi_D^B]=0\qquad \text{and} \qquad \E[\tilde{h}^B]=0.$$
\end{lemm}

\begin{lemm}
\label{remarkpropmarkov1}
The decomposition $h^D = \tilde{h}^B + \varphi_D^B$ given in the MTD is unique. Moreover, one can observe the following properties
\begin{itemize}
\item[i)] if $B$ and $B'$ are balls such that $B'\subset B \subset D$, then we have the Chasles relation 
$$\varphi_B^{B'} = \varphi_D^{B'}-\varphi_D^B = \tilde{h}^B - \tilde{h}^{B'},$$
and $\varphi_B^{B'}$ is a random distribution of compact support $\bar{B}$ corresponding to integrating with respect to a harmonic function when restricted to $B'$.
\item[ii)] if $B_1,\dots,B_n$ are $n$ disjoint balls, then for any $j\in\{1,\dots,n\}$,
$$\varphi := h^D - \sum_{i=1}^n\tilde{h}^{B_i} = \varphi_D^{B_j} - \sum_{i\neq j} \tilde{h}^{B_i},$$
is harmonic in $\cup_{i=1}^n B_i$ and equal to $\varphi_D^{B_j}$ on $B_j$. Moreover $\{\varphi,(\tilde{h}^{B_i})_{i\neq j}\}$ is measurable with respect to the $\sigma$-algebra generated by $\restriction{(h^D)}{D-B_j}$.
\item[iii)] we have that $\E[\varphi_D^{B'}(x)\varphi_D^B(x)]=\E[\varphi_D^B(x)^2]$ for balls $B' \subseteq B$
\end{itemize}
\end{lemm}

\begin{proof}
The uniqueness and first two properties follow exactly like in \cite{BPR18} in the 2D case, or as sketched in \cite{AruPowell} Section 2. The third property is immediate from the fact that $\varphi_D^B$ is measurable with respect to the $\sigma$-algebra generated by $\restriction{(h^D)}{D-B'}$\footnote{this means that $\varphi_D^B$ on $B$ is measurable with respect to the intersection of all the $\sigma$-algebras generated by the $\restriction{(h^D)}{B_0}$, for all the open sets $B_0$ containing $D-B'$. Moreover, being measurable with respect to the $\sigma$-algebra generated by the $\restriction{(h^D)}{B_0}$ means being measurable with respect to the $\sigma$-algebra generated by all the $h^D(f)$ where $\text{supp} f \in B_0$. See e.g. \cite{Aruthesis}}.
\end{proof}

\begin{lemm}
\label{remarkpropmarkov2}
If $B(x,\eps')\subset B(x,\eps)\subset D$, then
$$ \E\left[\varphi_D^{B(x,\eps)}(x)^2\right]\leq \E\left[\varphi_D^{B(x,\eps')}(x)^2\right].$$
Moreover, for every $\delta > 0$ there exists a constant $C(\delta)>0$ such that for any $x\in D$ with $\text{d}(x,\partial D)>\delta$ and any $\eps<\delta$ we have
$$\E\left[ \varphi_D^{B(x,\eps)}(x)^2 \right] \leq C(\delta) (1+s_d(\eps)),$$
where $s_d(r)$ is the function defined on $\R_{>0}$ by $\log(1/r)$  if $d=2$ and $|x|^{2-d}$ if $d\geq 3$.
\end{lemm}

Here the proof is a bit different to the one in \cite{AruPowell}, and we need to use the uniform bound on the second moment to conclude. 

\begin{proof}

Let $\eta$ be a rotationally symmetric mollifier with total mass one and support in the unit ball $\B \subseteq \R^d$ and define $\eta_x^\eps$ by $\eps^{-d}\eta(\tfrac{\cdot-x}{\eps})$. From Lemma \ref{remarkpropmarkov1} and the mean-value property of harmonic functions, we can write 
\item[~]
$$ \E\left[\varphi_D^{B(x,\eps')}(x)^2\right] = \E\left[\varphi_D^{B(x,\eps)}(x)^2\right] + \E\left[\left((\tilde{h}^{B(x,\eps)},\eta_x^{\eps'})-(\tilde{h}^{B(x,\eps')},\eta_x^{\eps'})\right)^2\right]$$
and the first point follows.\\
Now, define $2^{-n} \leq \eps < 2^{-(n-1)}$ and $n_0:=\inf\{m\leq n ~:~2^{-m}<\delta\}$. Then $\E\left[\varphi_D^{B(x,\eps)}(x)^2\right] \leq \E\left[\varphi_D^{B(x,2^{-n})}(x)^2\right] $.
Using the above identity several times we obtain
$$ \E\left[\varphi_D^{B(x,2^{-n})}(x)^2\right] = \E\left[\varphi_D^{B(x,2^{-n_0})}(x)^2\right] + \sum_{m=n_0+1}^{n-1}\E\left[\left( (\tilde{h}^{B(x,2^{-(m-1)})},\eta_x^{2^{-m}}) - (\tilde{h}^{B(x,2^{-m})},\eta_x^{2^{-m}}) \right)^2\right].$$
We can bound this by
$$\E\left[ (h^D,\eta_x^{2^{-n_0}})^2 \right] + \sum_{m=n_0+1}^{n-1}\E\left[ (\tilde{h}^{B(x,2^{-(m-1)})},\eta_x^{2^{-m}})^2\right].$$
The uniform bound on the 2nd moments (that follows from the uniform bound on the 4th moments, i.e. assumption B) bounds the first term by a constant that depends on $\delta$. \corr{Further, by scaling there is some universal constant $C$ $\sup_{z \in D}|\eta_x^{2^{-m}}(z)| \leq C2^{md}$. Thus the local bound on the second moments (that follows from the local bound on the 4th moments, i.e. assumption A) bounds each of the terms of the sum by $C^22^{4md} 2^{-(m-1)2(d+2)}=2^{2(d+2)}2^{2m(d-2)}$. Summing together}, we obtain the relevant bounds. \end{proof}

\noindent From similar arguments with some more care, one can also prove the analogous fourth moment estimate, see Appendix \ref{appendix_lemma4}. Here again we make use of the uniform bound on the 4th moments.

\begin{lemm}
\label{remarkpropmarkov3}
If $B(x,\eps')\subset B(x,\eps)\subset D$, then
$$ \E\left[\varphi_D^{B(x,\eps)}(x)^4\right]\leq \E\left[\varphi_D^{B(x,\eps')}(x)^4\right].$$
Moreover, for any $\delta>0$ there exists a constant $C(\delta)>0$ such that for any $x\in D$ with $\text{d}(x,\partial D)>\delta$ and any $\eps<\delta$ we have
$$\E\left[ \varphi_D^{B(x,\eps)}(x)^4 \right] \leq C(\delta) (1+s_d(\eps)^2)$$
and moreover
\begin{equation}\label{eq:4pt}
\E\left[  (h^D,\nu_{x_1}^\eps)(h^D,\nu_{x_2}^\eps)(h^D,\nu_{x_3}^\eps)(h^D,\nu_{x_4}^\eps) \right]^4 \leq C(\delta)\prod_{i=1}^4(1+\max_{i\neq j} s_d(|x_i-x_j|)^2).
\end{equation}
\end{lemm}

\subsubsection{The two-point function}

We define for $x \neq y\in D$, $B_1$ and $B_2$ two disjoint balls in $D$ containing respectively $x$ and $y$, the two-point function $k_D$ by
\begin{align}
\label{eqk}
 k_D(x,y) := \E\left[\varphi_D^{B_1}(x) \varphi_D^{B_2}(y)\right]. 
\end{align}
It is easy to verify that the function $k_D$ is well defined (see e.g. \cite{AruPowell}), which means here that $k_D(x,y)$ does not depend on the choice of $B_1$ and $B_2$. It also direct to see that $k_D(x,y)$ can be written as
$$k_D(x,y)=\E\left[(h^D,\eta_x^{\eps_1}) (h^D,\eta_y^{\eps_2})\right],$$
for any $\eps_1,\eps_2>0$ such that $B(x,\eps_1)$ and $B(y,\eps_2)$ are disjoint and both included in D.\\

In fact, one can almost word by word follow the proof of the similar statements in \cite{AruPowell} in Section 2 to show the following results.

\begin{lemma}
\label{lemmekbound}
For every $\delta<\text{d}(x,\partial D)$ there exists $C(\delta)>0$ such that for any $y\in D-\{x\}$ and $\text{d}(y,\partial D)>\delta$ it holds 
\begin{align}
    \label{def_k_bound}
    |k_D(x,y)|\leq C(\delta)(1+s_d(|x-y|)).
\end{align}
\end{lemma}

This just follows from Lemma \ref{remarkpropmarkov1} and Cauchy-Schwarz by choosing the disks $B_1, B_2$ to have radius $\|x-y\|/2$. Next, we verify that the two-point function is indeed the pointwise covariance-function:

\begin{lemma}
\label{lemme_k_cov_kernel}
For any $f,g\in C^\infty_c(D)$, we have
\begin{align}
\label{K_limit}
K_D(f,g):=\E\left[(h^D,f)(h^D,g)\right] = \int_D\int_D k_D(x,y)f(x)g(y)dxdy.
\end{align}
\end{lemma}
\noindent As in \cite{AruPowell} Section 2, the key step here is showing that
$$K_D(f,g) = \lim_{\eps\mapsto 0} \int_D \int_D f(x)g(y)K_D(\eta_x^\eps,\eta_y^\eps)dxdy.$$
But observe that 
$$K_D(f,g) = \E\left[\lim_{\eps\mapsto 0} (h^D,f\ast \eta_0^\eps)(h^D,g\ast \eta_0^\eps)\right]=\E\left[\lim_{\eps\mapsto 0} \int_D\int_D f(x)g(y)(h^D,\eta_x^\eps)(h^D,\eta_y^\eps)dxdy\right]$$ and thus to get the desired result, it is enough to show that the family $$\left(\int_D\int_D f(x)g(y)(h^D,\eta_x^\eps)(h^D,\eta_y^\eps)dxdy\right)_{\eps>0},$$ is uniformly integrable. For that one can control
$ \E\left[\left((h^D,f\ast \eta_0^\eps)(h^D,g\ast \eta_0^\eps)\right)^2\right]$ 
using \eqref{eq:4pt}. \\
A very similar argument by considering smooth mollifiers $\eta_{y_k}^\eps$ converging to $\eta_y^\eps$ gives that
\begin{lemma}\label{lem:ctyk2}
For any $x \in D$, we have that $k_D(x,y)$ is continuous in $D \setminus \{x\}$.
\end{lemma}
Finally, one can in fact identify the two-point function as the Green's function. We only need it in the case the function $a: D \to \R$ is non-constant, and can follow \cite{AruPowell} almost word by word.

\begin{prop}[Identifying the covariance]\label{prop:Green}
We have that $k_D(x,y) = bG_D(x, y)$, where $G_D$ is the zero boundary Green's function and $b > 0$ a constant.
\end{prop}

\begin{proof}
As in \cite{AruPowell}, we can use the following characterisation of $G_D(x, \cdot)$:
\begin{itemize}
\item $G_D(x,\cdot)$ is harmonic in $D-\{x\}$.
\item $y\mapsto G_D(x,y) - \Phi_d(|x-y|)$ bounded in a neighborhood of $x$.
\item $G_D(x,\cdot)$ has zero boundary conditions in the sense of Definition \ref{defiZB}
\end{itemize}
Using this, one can follow the proof in \cite{AruPowell} given for $x = 0$ for every $x \in D$, to conclude that for every $x \in D$, we have that $k_D(x,\cdot) = b(x) G_D(x,\cdot)$. 

Now, in \cite{AruPowell} the condition of stationary independent increments was used to deduce that $b$ is constant. Here we just observe that by definition $k_D(x,y)$ is symmetric as is $G_D(x,y)$: thus for any $x \neq y$, we have that $k_D(x, y) = b(x) G_D(x, y)$ equals $k_D(y,x) = b(y)G_D(x,y)$, and hence $b(x)$ is in fact constant. 
\end{proof}

\subsection{Definition of the dynamics and proof of Theorem \ref{mainthm}}
\label{section_dyn_def}

We introduce the following collection of dynamics on our field, indexed by a parameter $\eps > 0$, that for each $\eps$ and obtained by resampling the field in $\eps-$balls around uniform points at Poissonian times.

\begin{defn}
\label{defi_dyn}
Let $(\Omega,\Ff,(\Ff_t)_{t\geq 0},\mathbb{P})$ be a filtered probability space. Fix $\eps>0$, $D$ a regular domain, and let $h^D$ be a random distribution measurable with respect to $\Ff_0$ having all the properties listed in Theorem \ref{mainthm}.

\corr{Set $\kappa_d = \frac{\text{Vol}(B(0,1))}{d(d+2)}$ and let $(N_t^\eps)_{t\geq 0}$ be a Poisson clock with rate  $\t:= \frac{|D|}{\kappa_d}\eps^{-(2+d)}$ } with respect to the filtration $(\Ff_t)_{t\geq 0}$ and $(T^\eps_n)_{n\geq 0}$ the times when the Poisson clock rings: i.e $T^\eps_n := \inf\{t\geq 0 ~:~N^\eps_t = n\}$ for $n\geq 0$. Finally, let $(U_n^\eps)_{n \geq 1}$ be a sequence of uniformly distributed points in $D_\eps:=\{x\in D~:~d(x,\partial D)>2\eps\}$ measurable with respect to the filtration $(\Ff_{T^\eps_n})_{n\geq 1}$ and independent from $\Ff_{T^\eps_n-}$ and $B_n^\eps := B(U_n^\eps, \eps)$ balls of radius $\eps$ around the uniform points. 

We now define a cadlag stochastic process $(\g_D(t))_{t\geq 0}$ adapted to $(\Ff_t)_{t\geq 0}$ and with values in $\Ss_D$ as follows :
\begin{itemize}
\item[$\bullet$] $\g_D(0):=h^D$,
\item[$\bullet$] for every $n\geq 1$, at time $T_n^\eps$ resample the process $\g_D(T^\eps_{n-1})$ on $B^\eps_n$, using the MTD: i.e. we set
\begin{equation}
\label{resampling}
\g_D(T^\eps_n):= \g_D(T^\eps_{n-1}) + \hat{h}^{B^\eps_n}-\tilde{h}^{B^\eps_n},
\end{equation}
where $\hat{h}^{B^\eps_n}$ is measurable with respect to $\Ff_{T^\eps_n}$; conditionally independent from $\Ff_{T^\eps_n-}$ when conditioned on $\varphi_D^{B^\eps_n}$, and is such that $(\hat{h}^{B^\eps_n},\varphi_D^{B^\eps_n})$ has the same law as $(\tilde{h}^{B^\eps_n},\varphi_D^{B^\eps_n})$.
\item[$\bullet$] $\g_D(t)$ is constant in the intervals $[T^\eps_{n-1}, T^\eps_n)$.
\end{itemize}
\end{defn}

\begin{rema}
Notice that it is not a priori clear whether the dynamics is even well-defined. We have defined the MTD for any fixed $x$, but we would need to know if $x \to \tilde h^{B(x,\eps)}$ or alternatively, whether $x \to h^D - \varphi^{B(x,\eps)}_D$ is defined as a measurable map. We will show in Proposition \ref{prop:harmonic} that indeed there is a modification that makes the map measurable.
\end{rema}
\noindent The main theorem then states that this process converges in law to the stochastic heat equation.
\begin{theorem}
\label{thm_dyn}
\label{mainthmlemma}
\corr{If $\t:=\frac{|D|}{\kappa_d}\,\eps^{-(2+d)}$ with $\kappa_d$ as above}, then $(\g_D(t))_{t \geq 0}$ converges in law  to the stationary solution of \corr{$a\sqrt{\frac{2}{\kappa_d}}$-SHE} with zero boundary conditions, in the space $D(\R^+, \Ss)$ of distribution-valued cadlag processes, when $\eps\mapsto 0$.
\end{theorem}

We can now conclude the convergence result, i.e. Theorem \ref{mainthm}.

\begin{proof}[Proof of Theorem \ref{mainthm}]
Let $h^D$ be as in the statement and consider the dynamics $\g_D(t)$ defined just above starting from $h^D$. The MTD implies that all the marginals $\g_D(t)$ (for $t\geq 0$) follow the same law and thus $h^D$ is a stationary law of the above-defined dynamics for every $\eps > 0$. Further, by assumption it satisfies zero boundary conditions at all $t > 0$.

But by Theorem \ref{mainthmlemma}, we know that $\g_D(t)$ converges to a stationary solution of $\sqrt{\frac{2}{\kappa_d}}a$-SHE with zero boundary conditions. Proposition \ref{propGFFSHE} hence implies that $h^D$ is a zero mean Gaussian and the covariance is explicit.

Now there are two cases:
\begin{itemize}
\item Case 1: if $a(x)$ is constant, we have an actual GFF times a constant and we are already done.
\item Case 2: we need to show that in fact our MTD condition forces $a(x)$ to be constant. This is the content of Proposition \ref{prop:Green}.
\end{itemize}
\end{proof}

The proof of Theorem \ref{mainthmlemma} spans the next three subsections. In Section \ref{section_laplacien} we define an approximate Laplacian using the MTD and show that properly renormalized it converges to the Laplacian; we also obtain integration by parts identities; this section explains maybe best why the convergence should be true. In Section \ref{section_epsilon} we define the relevant martingales from the resampling dynamics and obtain their convergence. Finally in Section \ref{section_conv} we show the convergence of the resampling dynamics and conclude Theorem \ref{mainthmlemma}.\\
~\\
For the sake of readability, we will shorten some of the notations:
\begin{nota}\label{notations}
\begin{itemize}
\item Since $D\in\Dd$ is fixed by Theorem \ref{mainthmlemma}, we will from now on drop the dependency in $D$ of $\g_D:=(\g_D(t))_{t\geq 0}$, and rather denote it by $\g:=(\g_t)_{t\geq 0}$.
\item 
Also, since at every time $t\geq 0$ and every point $x\in D$ we have by the MTD a decomposition of $\g_t$ given by $\g_t = \tilde{h}^{B(x,\eps)} + \varphi_D^{B(x,\eps)}$ we will encode this with the notation $\g_t=\tilde{h}^\eps(t,x) + \varphi^\eps(t,x)$. 
\item For the specific case of the sequences involved in \eqref{resampling} $\tilde{h}^\eps(T^\eps_{n-1},U^\eps_n)$, $\varphi^\eps(T^\eps_{n-1},U^\eps_n)$ and $\hat{h}^\eps(T^\eps_n,U^\eps_n)$ (the latter being also denoted $\hat{h}^{B^\eps_n}$ in \eqref{resampling}), we will denote respectively $\tilde{h}^\eps_n$, $\varphi^\eps_n$ and $\hat{h}^\eps_n$. Finally, we will denote the composed process $U^\eps_{N^\eps_t}$ adapted to the filtration $(\Ff_t)_{t\geq 0}$ by $U^\eps_t$.
\end{itemize}
\end{nota}
Similarly, in what follows we will pretend as if $U_n^\eps$ follows the uniform probability law on $D$, forgetting about the boundary effects; it is easy to convince oneself that the errors coming from this go to zero as $\eps \to 0$.


\subsection{The map $x \to \varphi^{B(x,\eps)}$ and an approximation of $\Delta$}
\label{section_laplacien}
In this section we will define an operator which will approximate the operator $\Delta$ in \eqref{SHE}. We show its convergence to $\Delta$, prove an integration by parts identity that holds at any fixed $\eps > 0$, and finally prove its relation to the martingale decomposition of the field $h^D$. Before doing this, we lay ground by connecting our harmonic extension to the usual Poisson kernel.\\
~\\
For each ball $B(x,\eps)\subset D$, recall the Poisson kernel defined by
\begin{align}
\label{def_poisson_kernel}
\begin{array}{ccccc}
	P_{B(x,\eps)} & : &  B(x,\eps) \times S(x,\eps) & \longrightarrow & \R \\
	& & (z,\theta) & \longmapsto & \dfrac{\eps^2-|z-x|^2}{A_d\eps|z-\theta|^d},
\end{array}
\end{align} 
where $A_d$ is the surface of the unit sphere $\mathbb{S}$.
Then, if $g$ is a continuous function on $D$, we can define the Poisson operator  by
\begin{align}
\label{def_poisson_int}
P_{B(x,\eps)}[g](z) := \int_{S(x,\eps)} g(\theta)P_{B(x,\eps)}(z,\theta)d\theta \qquad \text{    for }z\in B(x,\eps).
\end{align}
\corr{It is also useful to further extend this definition to $z \in B(x,\eps)^c$ by $P_{B(x,\eps)}[g](z) = g(z)$.}
For more details on the Poisson kernel and proofs of its properties, one can look at \cite{axlerharm}. \\

We start off by verifying that the harmonic part of the MTD is really given by a Poisson kernel applied to the GFF on the boundary of the relevant ball - this is not obvious from our decomposition!

\begin{lemm}
\label{lemma_Gamma_approx_laplacian}
Fix some $\eps > 0$ and let $(\alpha_n)_{n\geq 1}$ be a sequence of smooth mollifiers approximating the identity with support contained in the disk of radius $2^{-n}$.
There is a universal subsequence $(n_m)_{m \geq 1}$ such that for all $x \in D$ with $d(x, \partial D) > 2\eps$ we have
$$ P_{B(x,\epsilon)}[h^D\,\ast\,\alpha_{n_m}] \mapsto \varphi^{B(x,\epsilon)}_D, \qquad \text{when } m\mapsto \infty, $$
almost surely in $L^2(B(x, \epsilon))$. 
\end{lemm}

\begin{proof}[Proof of Lemma \ref{lemma_Gamma_approx_laplacian}]
Fix some $x \in D_\eps$. First, notice that $P_B[h^D\ast\alpha_n](z) = (h^D,\alpha_n\ast P_B(z,d\theta))$ for all $z\in B$. Using this, \eqref{K_limit} and the two-point estimate \eqref{def_k_bound} one can explicitly calculate that for any $m > n$,
$$\E \| P_{B(x,\epsilon)}[h^D\,\ast\,\alpha_n] -  P_{B(x,\epsilon)}[h^D\,\ast\,\alpha_m]\|_2^2 = o_n(1)$$
uniformly for $x \in D_\eps$. 
In particular, there is an universal subsequence $(n_m)_{m \geq 1}$ such that 
$$\E \| P_{B(x,\epsilon)}[h^D\,\ast\,\alpha_{n_m}] -  P_{B(x,\epsilon)}[h^D\,\ast\,\alpha_{n_l}]\|_2^2 \leq 2^{-m}$$
for all $l \geq m$ for every $x \in D_\eps$. By Borel-Cantelli, along such a subsequence the functions $P_{B(x,\epsilon)}[h^D\,\ast\,\alpha_{n_m}]$ converge in $L^2(B(x,\eps))$ almost surely to some $\tilde \varphi_{x, \eps}$. Moreover, the limit $\tilde \varphi_{x, \eps}$ is weakly harmonic in $B(x,\eps)$ and thus coincides with a harmonic function almost everywhere. We claim that  $\tilde \varphi_{x, \eps} = \varphi_D^{B(x,\eps)}$ almost surely.

Now fix $z \in B(x,\eps)$. Then, there exists a sequence $(f_{m})_{m\ge 0}$ of functions satisfying the zero boundary condition assumption \ref{defiZB} such that $(\varphi,f_{m})=\varphi(z)$ for all $m$ and any harmonic function $\varphi$ in $B(x,\eps)$. Now applying this to the MTD for $B=B(x, \eps)$ and using the fact that $\tilde h^B$ satisfies zero boundary conditions in the sense of Definition \ref{defiZB} in $B$, we see that $$\lim_{m \to \infty}\E\left(|(\varphi_D^B, f_{m}) - (h^D, f_{m})|^2\right) = 0.$$ \corr{The final ingredient is the following}  
\begin{align}
\label{lemma_2_17}
\lim_{m \to \infty}\E\left(|(h^D, f_m)- (P_{B(x,\epsilon)}[h^D\,\ast\,\alpha_{n_m}], f_m)|^2\right) = 0,
\end{align}
\corr{Indeed, using the definition of $\tilde \varphi_{x,\eps}$, \eqref{lemma_2_17} and the equation just above it  allow us to} conclude that  $\tilde \varphi_{x,\eps}(z) = \varphi_D^B(z)$. Applying this for a dense collection of $z$ and using continuity of $\varphi_D^{B(x,\eps)}$ and $\tilde \varphi_{x,\eps}$ (coming from their harmonicity) proves the claim and thus the lemma.\\
~\\
\corr{
Thus it remains to justify \eqref{lemma_2_17}. Without loss of generality, and to simplify the calculations, we choose the functions $f_m$ so that the support of $f_m$ is at distance at least $2^{-n_m}$ from the boundary of $B(x,\eps)$. Set
\[
q_{m,z}(u):=
\big(\alpha_{n_m}\ast P_B(z,\cdot)\big)(u).
\]
Then
\[
P_B[h^D\ast\alpha_{n_m}](z)
=
(h^D,q_{m,z}).
\]
We compute
\begin{align*}
& \E\left[
\left|
(h^D, f_m)
-
(P_{B(x,\eps)}[h^D\,\ast\,\alpha_{n_m}], f_m)
\right|^2
\right] \\
=~&
\E\left[(h^D,f_m)^2\right]
-
2\int_B f_m(z_0)
\E\left[
(h^D,f_m)(h^D,q_{m,z_0})
\right]dz_0 \\
&\quad
+
\int_B\int_B f_m(z_1)f_m(z_2)
\E\left[
(h^D,q_{m,z_1})(h^D,q_{m,z_2})
\right]dz_1dz_2 .
\end{align*}
Using Lemma \ref{lemme_k_cov_kernel}, this equals
\begin{align*}
&\int_B \int_B f_m(z_1)f_m(z_2)
\Bigg[
k_D(z_1,z_2)
-
2\int_D k_D(u,z_2)q_{m,z_1}(u)\,du \\
&\hspace{4.8cm}
+
\int_D \int_D
k_D(u,v)q_{m,z_1}(u)q_{m,z_2}(v)\,du\,dv
\Bigg]dz_1dz_2 .
\end{align*}
Here, to get the last equality, we computed the expectations using Lemma \ref{lemme_k_cov_kernel} and factored out the test functions. Finally, using the continuity of the kernel $k_D$ from Lemma \ref{lem:ctyk2}, together with the bound of Lemma \ref{lemmekbound}, we see that the entire quantity goes to zero as $m\to\infty$.}
\end{proof}
This lemma allows us to define a measurable map $x \to \varphi^{B(x,\epsilon)}_D$ via approximation, and moreover its proof also gives the convergence of the integral:

\begin{prop}[Definition of $x \to \varphi^{B(x, \epsilon)}_D$]\label{prop:harmonic}
For every $\eps > 0$, there is a modification of $x \to \varphi^{B(x, \eps)}$ on $D_\eps$ with values in $S_D$ that is measurable. We abuse the notation and still write it as $x \to \varphi^{B(x, \eps)}$.

Moreover, if we let $(\alpha_m)_{m\geq 1}$ be the sequence of smooth mollifiers approximating the identity coming from Lemma \ref{lemma_Gamma_approx_laplacian}.
Then we have
$$ \int_D (P_{B(x,\eps)}[h^D\,\ast\,\alpha_n], \cdot) dx \mapsto \int _D (\varphi^{B(x,\eps)}_D, \cdot) dx \qquad \text{when } n\mapsto \infty, $$
almost surely in $\Ss_{D}$\footnote{the convergence in $\Ss_D$ is the pointwise convergence (i.e $h_n \mapsto_n h$ is defined by $h_n(f)\mapsto_n h(f)$ for every $f$).}. 
\end{prop}

\begin{proof}
For any $n \geq 1$ the function $x \to P_{B(x,\epsilon)}[h^D\,\ast\,\alpha_n] $ is measurable. Moreover by the lemma above we can choose some universal sequence of mollifiers $(\alpha_n)_{n \geq 1}$ such that for every $x \in D_\eps$ almost surely, $P_{B(x,\epsilon)}[h^D\,\ast\,\alpha_n] \mapsto \varphi^{B(x,\epsilon)}_D$. Thus by Fubini we have that almost surely the functions $x \to P_{B(x,\epsilon)}[h^D\,\ast\,\alpha_n]$ converge Lebesgue almost everywhere on $D_\eps$. Hence one can define a limiting function that is almost surely and almost everywhere limit, and in particular is a modification of $x \to \varphi^{B(x, \eps)}$.

The proof of the second part is basically the same as for the lemma above. It suffices to show that for any $f \in C^\infty(\R^d)$ we have that almost surely
\begin{align}
    \label{conv_mollified_delta_obj}
 \int_D (P_{B(x,\eps)}[h^D\,\ast\,\alpha_m],f)dx \mapsto \int_D(\varphi^{B(x,\eps)}_D,f) dx, \qquad \text{when } m\mapsto \infty.
\end{align}
But now, applying first Jensen inequality and then Fubini-Tonelli since the integrand is positive, we obtain
$$
\E\left[\left(\int_D (P_{B(x,\eps)}[h^D\,\ast\,\alpha_m],f)dx - \int_D(\varphi^{B(x,\eps)}_D,f) dx \right)^2 \right] \leq |D|\int_D dx \E \left[\left((P_{B(x,\eps)}[h^D\,\ast\,\alpha_m] - \varphi^{B(x,\eps)}_D,f)\right)^2 \right]
$$
which can be bounded by $|D|2^{-m}$ by the choice of $\alpha_m$ in the proof of Lemma \ref{lemma_Gamma_approx_laplacian}, where $|D|$ is the Lebesgue measure of a domain $D$. Thus we get the desired almost sure convergence by the Borel-Cantelli lemma.
\end{proof}

\subsubsection{An approximation of the Laplacian}

Recall $\kappa_d = \frac{\text{Vol}(B(0,1))}{d(d+2)}$. We now observe the following approximation of the Laplacian:
\corr{
\begin{defi}[Approximation of $\Delta$]
For any fixed $\eps>0$ we define the operator $\Delta_\eps ~:~C^1(\R^d) \longrightarrow C^1(\R^d)$\footnote{$C^1(\R^d)$ is the space of continuously differentiable functions on $\R^d$.} by 
\begin{align}
\label{Delta_epsilon}
\Delta_\eps g(z) := \dfrac{\eps^{-d-2}}{\kappa_d}\int_{B(z,\eps)} \left(P_{B(x,\eps)}[g](z)-g(z)\right)dx.
\end{align}
Equivalently for any test function $f \in C_c^\infty(\R^d)$ we set
\begin{equation}
\label{Delta_epsilon2}
(\Delta_\eps g, f) := \dfrac{\eps^{-(d+2)}}{\kappa_d}\int_D\left(\int_{B(x,\eps)} \left(P_{B(x,\eps)}[g](z)f(z)-g(z)f(z)\right)dz\right)dx.
\end{equation}
\end{defi}}
Notice that a corollary of Proposition \ref{prop:harmonic} is that we can define $\Delta_\eps h^D$ as a limit of $\Delta_\eps h^D \ast \alpha_m$ for the sequence described in the proposition.

The key result of this section is then the following proposition. 
\begin{prop}[Convergence of $\Delta_\eps$ to $\Delta$]
\label{prop_delta_eps}
\corr{For any family $(g)_{g\in\Gg}$ in $C^\infty_c(\R^d)$ bounded in $\Ss_D$\footnote{Here boundedness in $\Ss_D$ means boundedness as a family of distributions: for every bounded set $K\subset C_c^\infty(D)$ one has $\sup_{T\in\mathcal B,\phi\in K}|(T,\phi)|<\infty$. By the uniform boundedness principle this is equivalent to pointwise boundedness on test functions.}, we have, in $\Ss_D$,}
$$ \Delta_\eps g \longmapsto \Delta g, \qquad \text{as }\eps\mapsto 0.$$
Moreover, this convergence is pointwise uniform: for any fixed $f\in C^\infty_c(D)$,
$$\sup_{g\in \Gg} |(\Delta_\eps g - \Delta g,f)| \mapsto 0,\qquad \text{as }\eps\mapsto 0.$$ 
\end{prop}

\corr{
Note that
    $\Delta_\eps g$ is probably not the most natural approximation of $\Delta g$ that one would expect in an analysis context, and that might also be the reason why we did not manage to find such an approximation in the literature. A very close cousin would be of course well known: it would amount to taking the average around an epsilon minus the central value, all renormalized $\eps^{-2}\left(\frac{1}{|\partial B(z,\eps)|}\int_{\partial B(z,\eps)}g(x)dx - g(z)\right)$. Our approximation is not very far in spirit.
}

\begin{proof}
\corr{It is enough to test against a fixed $f\in C_c^\infty(D)$. For all sufficiently small $\eps$ the $2\eps$-neighbourhood of $\operatorname{supp} f$ is still contained in $D$, so the boundary of $D$ plays no role. Let $B_x=B(x,\eps)$. For $g\in C_c^\infty(\R^d)$, Green's identity in $B_x$ gives}
\[
\int_{B_x} f(z)\big(P_{B_x}[g](z)-g(z)\big)\,dz
=\int_{B_x}\Delta g(z)\int_{B_x}G_{B_x}(z,y)f(y)\,dy\,dz.
\]
\corr{Consequently, using \eqref{Delta_epsilon2} and changing the order of integration,}
\[
(\Delta_\eps g,f)=\dfrac{\eps^{-(d+2)}}{\kappa_d}\int_D\Delta g(z)
\int_{B_z}\int_{B_x}G_{B_x}(z,y)f(y)\,dy\,dx\,dz.
\]
\corr{We split $f(y)=f(z)+(f(y)-f(z))$. The first term is exactly normalized by our choice of $\tau_\eps$. Indeed, by translation invariance,
\[
\int_{B_z}\int_{B_x}G_{B_x}(z,y)\,dy\,dx
=\int_{B_0}\int_{B_u}
G_{B_u}(0,v)\,dv\,du.
\]
But now the inner integral can be explicitly computed to be equal to $\frac{\eps^2-\|u\|^2}{2d}$. Thus doing the outer integral we obtain exactly $\eps^{d+2}\kappa_d$
and therefore
\[
\dfrac{\eps^{-(d+2)}}{\kappa_d}\int_D\Delta g(z)f(z)
\int_{B_z}\int_{B_x}G_{B_x}(z,y)\,dy\,dx\,dz
=(\Delta g,f).
\]
It remains to control the error term
\[
\int_D\Delta g(z)R_\eps^f(z)\,dz,\qquad
R_\eps^f(z):=\dfrac{\eps^{-(d+2)}}{\kappa_d}\int_{B_z}\int_{B_x}G_{B_x}(z,y)(f(y)-f(z))\,dy\,dx.
\]}
\corr{For this purpose set, for $w\in B(0,2\eps)$,}
\[
\widetilde c_\eps(|w|):=\int_{B(0,\eps)\cap B(w,\eps)}G_{B_u}(0,w)\,du.
\]
This is the quantity obtained from $\int_{B_y\cap B_z}G_{B_x}(z,y)\,dx$ after translating $z$ to the origin and writing $w=y-z$. By rotational invariance it depends only on $|w|$, and by scaling $\widetilde c_\eps(r)=\eps^2\widetilde c_1(r/\eps)$. The singularity of the Green function is integrable, hence the integrals $\int_0^2 s^{d-1+k}\widetilde c_1(s)\,ds$ that occur below are finite. With $w=y-z$ we can write
\[
R_\eps^f(z)=\dfrac{\eps^{-(d+2)}}{\kappa_d}\int_{B(0,2\eps)}(f(z+w)-f(z))\widetilde c_\eps(|w|)\,dw.
\]
Passing to polar coordinates gives
\[
R_\eps^f(z)=\dfrac{\eps^{-(d+2)}}{\kappa_d}\int_0^{2\eps}r^{d-1}\widetilde c_\eps(r)
\int_{\mathbb S}\big(f(z+r\theta)-f(z)\big)\,d\theta\,dr.
\]
\corr{The linear term in the Taylor expansion of $f(z+r\theta)-f(z)$ integrates to zero on the sphere, while the second-order term is $\frac{A_d}{2d}r^2\Delta f(z)$ and the remainder is $O_f(r^3)$, uniformly together with all derivatives in $z$. More explicitly, for a smooth function we can write the reminder}
\begin{align}
\label{reste_taylor}
R_{f}(z\,;r\theta):=\sum_{|\beta|=3}\frac{3}{\beta!}(r\theta)^\beta\int_0^1(1-t)^2\partial_\beta f(z+tr\theta)\,dt.
\end{align}
\corr{Since $\widetilde c_\eps(r)=\eps^2\widetilde c_1(r/\eps)$, this yields, for every seminorm $\|\cdot\|_{C^k}$,}
\[
\|R_\eps^f\|_{C^k}\le C_{f,k}\eps^{-d-2}
\int_0^{2\eps}r^{d+1}\widetilde c_\eps(r)\,dr+O_{f,k}(\eps^3)
\le C'_{f,k}\eps^2.
\]
\corr{In particular $R_\eps^f\to0$ in $C_c^\infty(D)$. Hence, uniformly for $g$ in any bounded subset of $\Ss_D$,}
\[
\left|\int_D\Delta g(z)R_\eps^f(z)\,dz\right|
= |(g,\Delta R_\eps^f)|\longrightarrow0.
\]
\corr{This proves both the convergence and the stated uniformity.}
\end{proof}
As a Corollary we obtain the following:
\begin{coro}\label{cor:deltah}
Suppose $(h_n)_{n \geq 1}$ have the law of $h^D$ and converge to $h^D$ in $\Ss_D$ almost surely.
Then for every $f \in C_c^\infty(D)$ we have that $(\Delta_\eps h_n, f) \to (\Delta h^D, f)$ as $(n, \eps) \to (\infty, 0)$\footnote{i.e. independently of how $n$ approaches $\infty$ and $\eps$ approaches $0$} almost surely.
\end{coro}
\begin{proof}
To apply the proposition, we pick again the sequence of mollifiers $(\alpha_m)_{m \geq 1}$ given by Proposition \ref{prop:harmonic} and aim to prove that for any fixed test function $f$ it holds that $\lim_{(n,\epsilon,m)\mapsto (\infty,0,\infty)}(\Delta_\epsilon (h_n \,\ast\,\alpha_m),f) = (\Delta h^D,f)$, where the triple limit means that for all $\delta > 0$ we can choose $\eps^{-1}, n, m$ large enough so that $|(\Delta_\epsilon (h_n \,\ast\,\alpha_m),f) - (\Delta h^D,f)| < \delta$. 

First notice that 
$$\lim_{(n,m)\mapsto (\infty,\infty)}|(\Delta (h_n \,\ast\,\alpha_m),f) - (\Delta h^D,f)|=\lim_{(n,m)\mapsto (\infty,\infty)}|((h_n\,\ast\,\alpha_m),\Delta f) - (h^D,\Delta f)| =0.$$
Hence it is enough to prove that $$\lim_{(n,\epsilon,m)\mapsto (\infty,0,\infty)}|(\Delta_\epsilon (h_n \,\ast\,\alpha_m),f) - (\Delta (h_n \,\ast\,\alpha_m),f)|=0.$$ 
But $\Gg:=\{(h_n \,\ast\,\alpha_m)~:~n\geq 0,m\geq 1\}$ is a bounded set in $\Ss_D$ and thus it follows from Proposition \ref{prop_delta_eps} and the Moore-Osgood theorem.\\
We conclude that
\begin{align*}
(\Delta h^D,f) &= \lim_{(n,\epsilon,m)\mapsto (\infty,0,\infty)}(\Delta_\epsilon (h_n \,\ast\,\alpha_m),f) = \lim_{(n,\eps)\mapsto (\infty,0)}\lim_{m\mapsto \infty}(\Delta_\epsilon (h_n \,\ast\,\alpha_m),f) = \lim_{(n,\eps)\mapsto (\infty,0)}(\Delta_\epsilon h_n,f)
\end{align*}
where the last equality comes from Proposition \ref{prop:harmonic}, and concludes the proof.
\end{proof}

This approximate Laplacian satisfies also an integration by parts identity:
\begin{prop}[An integration by parts identity]
\label{prop_delta_commute}
    For every $g\in C^0(\R^d)$, $f\in C^\infty_c(D)$ and \corr{$0 < \eps < \text{dist}(\text{supp }f, \partial D)$} we have $(\Delta_\eps g,f) = (g,\Delta_\eps f)$. 
\end{prop}
\begin{proof}
    By linearity of the integral and the fact that $(g,f)=(f,g)$ it is enough to prove that
    \begin{align}
    \label{obj_delta_commute}
        \int_D (P_{B(x,\eps)}[g],f) dx = \int_D (P_{B(z,\eps)}[f],g)dz.
    \end{align}
To this end, we compute
\begin{align*}
    \int_D (P_{B(x,\eps)}[g],f) dx &= \int_D \int_{B(x,\eps)} \int_{S(x,\eps)} g(\theta) P_{B(x,\eps)}(z,\theta)d\theta f(z)dz dx \\
    &= \eps^{d-1}\int_{S(0,1)} \int_D \int_D \1_{\{|z-x|<\eps\}} g(x+\eps\theta) P_{B(x,\eps)}(z,x+\eps\theta) f(z)dz dx d\theta.
\end{align*}
Noting that by \eqref{def_poisson_kernel} it holds $P_{B(x,\eps)}(z,x+\eps\theta) = P_{B(z+\eps\theta,\eps)}(x+\eps\theta,z)$, we obtain with a double affine substitution
\begin{align*}
    \int_D (P_{B(x,\eps)}[g],f) dx &= \eps^{d-1}\int_{S(0,1)} \int_D \int_D \1_{\{|z-x|<\eps\}} g(x) P_{B(z,\eps)}(x,z-\eps\theta) f(z-\eps\theta)dz dx d\theta\\
    &=\int_D (P_{B(z,\eps)}[f],g)dz
\end{align*}
Thus, \eqref{obj_delta_commute} holds.
\end{proof}

Again there is a natural corollary, whose proof goes by combining the proposition above with the mollification from Proposition \ref{prop:harmonic}:
\begin{coro}\label{cor:intbyph}
Let $f\in C^\infty_c(D)$ and \corr{fix $0 < \eps < \text{dist}(\text{supp }f, \partial D)$}. We have $(\Delta_\eps h^D,f) = (h^D,\Delta_\eps f)$.
\end{coro}


\subsection{The martingales associated to $\g$ and their convergence}
\label{section_epsilon}
In this section, we define the martingales associated to the process $\g$ for any $\eps>0$ fixed and also determine their quadratic variation. 

\begin{prop}\label{prop:martingales}
Let $\g_t$ be as in Theorem \ref{mainthmlemma}, let $f \in C_c^\infty(\R^d)$. Then for every $\eps > 0$,  
\begin{align}
\label{MP1_epsilon}
(M_t^\eps,f) := (h^\eps_t,f) - \dfrac{\t}{|D|} \int_0^t \left( \int_D (\varphi^\eps(s,x),f)-(h^\eps_s,f)dx \right)ds =   (h^\eps_t,f) + \dfrac{\t}{|D|} \int_0^t \left( \int_D (\tilde{h}^\eps(s,x),f)dx \right) ds
\end{align}
is an $\Ff_t$-martingale. Moreover also
\begin{align}
\label{MP2_epsilon}
(Q_t^\eps,f) &:= (M_t^\eps,f)^2 - \dfrac{\t}{|D|}  \int_0^t \int_D \E[(\tilde{h}^\eps(s,x),f)^2~|~\varphi^\eps(s,x)]+(\tilde{h}^\eps(s,x),f)^2 dx ds,
\end{align} 
is an $\Ff_t$-martingale. In particular, the
quadratic variation of $(M_t^\eps,f)$ is given by
\begin{equation*}
(A^\eps_t,f):=(\langle M^\eps\rangle_t,f) = \dfrac{\t}{|D|} \int_0^t \int_D \E[(\tilde{h}^\eps(s,x),f)^2~|~\varphi^\eps(s,x)]+(\tilde{h}^\eps(s,x),f)^2 dx ds.
\end{equation*}
\end{prop}

\begin{proof}
Using the definition of the Poissonian dynamics, we can directly compute for $\delta:=t-s>0$,
\begin{align*}
\E[(h^\eps_t,f)~|~\mathcal{F}_s] &= (1-\t \delta)\E[(h^\eps_t,f)~|~\mathcal{F}_s,~N_t^\eps=N^\eps_s]+ \t \delta \,\E[(h^\eps_t,f)~|~\mathcal{F}_s,~N_t^\eps=N^\eps_s+1]+O_{\eps,f}(\delta^2) \\
&= (1-\t \delta) (h^\eps_s,f) + \dfrac{\t \delta}{|D|}\int_D \E[(h^\eps_t,f)~|~\mathcal{F}_s,~N_t^\eps=N^\eps_s+1,~U_t^\eps=x]dx+O_{\eps,f}(\delta^2)\\
&= (h^\eps_s,f) + \dfrac{\t \delta}{|D|}\int_D \E[(\tilde{h}^\eps(t,x),f)+(\varphi^\eps(t,x),f)-(h^\eps_s,f)~|~\mathcal{F}_s,~N_t^\eps=N^\eps_s+1,~U_t^\eps=x]dx+O_{\eps,f}(\delta^2).
\end{align*}
Here the $\delta^2$ error comes from the probability of having at least two Poisson ticks in the timeframe of $\eps$. Note that on the event $\{N^\eps_t=N^\eps_s+1, \, U^\eps_t=x\}$ it holds that $\varphi^\eps(t,x)=\varphi^\eps(s,x)$, and $\varphi^\eps(s,x)$ and $\g_s$ are measurable with respect to $\Ff_s$. Moreover, by definition of the dynamics, the MTD condition gives us that $\E[\tilde{h}^\eps(t,x)~|~\Ff_s,\,N^\eps_t=N^\eps_s+1, \, U^\eps_t=x] = 0$.  Hence, we obtain 
\begin{align}
\label{MP1_epsilon_intermediaire_1}
\E[(h^\eps_t,f)~|~\mathcal{F}_s] &= (h^\eps_s,f) + \dfrac{\t \delta}{|D|}\int_D (\varphi^\eps(s,x),f)-(h^\eps_s,f)dx+O_{\eps,f}(\delta^2).
\end{align}
We can now apply Lemma \ref{lemme_martingale} (stated and proven in the appendix) to deduce that 
$$
(M_t^\eps,f) := (h^\eps_t,f) - \dfrac{\t}{|D|} \int_0^t \left( \int_D (\varphi^\eps(s,x),f)-(h^\eps_s,f)dx \right)ds  = (h^\eps_t,f) + \dfrac{\t}{|D|} \int_0^t \left( \int_D (\tilde{h}^\eps(s,x),f)dx \right)ds$$
is a martingale. Indeed, the first two conditions of the lemma hold immediately and for the third condition we can verify 
\begin{align}
\label{tool_exp2}
\E\left[\sup_{t\leq T} \left| \dfrac{\t}{|D|}\int_D (\tilde{h}^\eps(t,x),f)dx \right|\right] &\leq \E\left[ \sum_{n=0}^{N^\eps_T} \left| \dfrac{\t}{|D|}\int_D (\tilde{h}^\eps(T^\eps_n,x),f) dx \right| \right] \nonumber\\
&\leq \sum_{n=0}^\infty e^{-\t T}\dfrac{(\t T)^n}{n!} \dfrac{\t}{|D|}\int_D \sqrt{\E\left[(\tilde{h}^\eps(T^\eps_n,x),f)^2\right]}dx \nonumber \\
&\leq C_a\sup_{z\in D}|f(z)|\t e^{-\t T}\eps^{1+d/2} < \infty,
\end{align}
for any fixed $\eps > 0$. Thus we conclude the first part of the proposition. \\
~\\
For the second part, let us drop the $f$-dependency for the sake of readability and again set $\delta = t -s > 0$. We again aim to use Lemma \ref{lemme_martingale} and hence aim for errors of order $\delta$. We have
\begin{equation*}
    \E[(M_t^\eps)^2~|~\Ff_s]-(M_s^\eps)^2 = \E\left[\left(\g_t - \g_s + \dfrac{\t}{|D|}\int_s^t  \int_D \tilde{h}^\eps_r dx dr \right)^2|\Ff_s\right].
\end{equation*}
First, notice that when opening the bracket the term $\E\left[\left(\dfrac{\t}{|D|}\int_s^t  \int_D \tilde{h}^\eps_r dx dr \right)^2|\Ff_s\right]$ is of order $O_{\eps, f}(\delta^2)$ - indeed the second moment of $\tilde{h}^\eps_r = (\tilde{h}^\eps(r,x),f)$ is bounded uniformly in $x, r$ by the MTD second moment condition. Thus the double integral over $t-s = \delta$ gives the error of $O(\delta^2)$ and we can safely ignore this term and concentrate on $\E\left[\left(\g_t - \g_s \right)^2|\Ff_s\right]$.

Using the definition of the Poissonian dynamics we can as above write
\begin{equation*}
\E\left[\left(\g_t - \g_s \right)^2|\Ff_s\right] = \t \delta \E\left[\left(\g_t - \g_s \right)^2|~\mathcal{F}_s,~N_t^\eps=N^\eps_s+1\right]+O_{\eps,f}(\delta^2),
\end{equation*}
and concentrate on the first term. We have 
$$\E\left[\left(\g_t - \g_s \right)^2|~\mathcal{F}_s,~N_t^\eps=N^\eps_s+1\right]
 = \E\left[\left(\tilde{h}^\eps(t, U_t^\eps) - \tilde{h}^\eps(s, U_t^\eps)\right)^2|~\mathcal{F}_s,~N_t^\eps=N^\eps_s+1\right] .$$ 
 Conditioning further on the uniform point $U_t^\eps$, we see that the cross-terms in opening the square bracket disappear thanks to the MTD first moment condition. Thus we are left 
 with 
$$\E\left[\left(\tilde{h}^\eps(t, U_t^\eps)\right)^2 + \left(\tilde{h}^\eps(s, U_t^\eps)\right)^2|~\mathcal{F}_s,~N_t^\eps=N^\eps_s+1\right] = \frac{1}{|D|}\int_D \E[(\tilde{h}^\eps(s,x),f)^2~|~\varphi^\eps(s,x)]+(\tilde{h}^\eps(s,x),f)^2 dx ds,$$
where we further used the fact that by definition of our resampling, under the conditioning on $(\mathcal{F}_s,~N_t^\eps=N^\eps_s+1, U_t^\eps = x)$, the law of $\tilde{h}^\eps(t, x)$ equals the law of  $\tilde{h}^\eps(s, x)$ conditioned on $\varphi^\eps(s,x)$.

We conclude that 
$$ \E[(M_t^\eps)^2~|~\Ff_s]-(M_s^\eps)^2 = \delta \dfrac{\t}{|D|}\int_D \E[(\tilde{h}^\eps(s,x),f)^2~|~\varphi^\eps(s,x)]+(\tilde{h}^\eps(s,x),f)^2 dx ds + O_{\eps, f}(\delta^2).$$
Thus verifying that the conditions I, II, III hold as above, we can apply Lemma \ref{lemme_martingale} to deduce that  
$$(Q_t^\eps,f) := (M_t^\eps,f)^2 - \dfrac{\t}{|D|}  \int_0^t \int_D \E[(\tilde{h}^\eps(s,x),f)^2~|~\varphi^\eps(s,x)]+(\tilde{h}^\eps(s,x),f)^2 dx ds$$
is a $\Ff_t-$martingale. The claim on the quadratic variation then follows.
\end{proof}

\subsubsection{The martingale central limit theorem}
\label{section_EK}
In this section we will prove that the sequence of martingales $(M_t^\eps-M_0^\eps)_{\eps>0}$ converges in the Skorokhod space towards a continuous centered Gaussian process with variance $\E[(M_t,f)^2]=t\langle af~|~a f\rangle$.

\begin{prop}\label{prop:mCLT} For any $f\in C_0^\infty(D)$, the martingales $(M^\eps_t-M_0^\eps,f)$ defined above converge in law in the Skorokhod space $D([0,\infty), \R)$ to a continuous centered Gaussian process $M_t$ with variance $\E[(M_t,f)^2]=\frac{2}{\kappa_d}t\langle af~|~a f\rangle$.
\end{prop}

To prove this proposition, we will use the following martingale central limit theorem \cite[Thm~1.4,~p.339]{EthierKurtz}. We state here the version of the theorem we used for the convenience of the reader:

\begin{theorem}[Martingale CLT] \label{thm:mCLT}
Let $(\Omega,\Ff,(\Ff_t)_{t\geq 0},\P)$ be a filtered probability space and $(M^\eps)_{\eps>0}$ a family of $\F_t$-local martingales with sample paths in $D([0,\infty);\R)$ and such that $M^\eps_0=0$. Let $(A^\eps)_{\eps>0}$ be a nonnegative and increasing process with simple paths in $D([0,\infty);\R)$. Suppose that $(M_t^\eps)^2 - A^\eps_t$ is a local $\Ff_t$-martingale and that
\begin{align}
    &\lim_{\eps\mapsto 0}\E\left[ \sup_{s\leq t} |A^\eps_s-A^\eps_{s-}  |\right]=0,\nonumber\\
    \label{EK_1.16}
    &\lim_{\eps\mapsto 0}\E\left[ \sup_{s\leq t} |M^\eps_s-M^\eps_{s-}|^2\right]=0\\
    \label{EK_1.19}
    &A^\eps_t~~ \longmapsto_{\eps\mapsto 0}~~ c t ,
\end{align}
where the last convergence holds in probability for all $t\geq 0$ and a certain $c>0$.\\
Then $M^\eps$ converges in law towards a process $(X_t)_{t\geq 0}$ of sample paths in $C([0,\infty);\R)$ with independent Gaussian increments, and such that $X_t$ and $X_t^2 - c t$ are martingales.
\end{theorem}

\begin{proof}[Proof of Proposition \ref{prop:mCLT}]
We will just verify that the conditions of Theorem \ref{thm:mCLT} are verified.

First, we already showed in Section \ref{section_epsilon} that for every $f \in C_c^\infty(D)$ the process $((M^\eps_t,f)-(M^\eps_0,f))$ is a c\`adl\`ag martingale equal to 0 at $t=0$ and whose quadratic covariation is given by the process $(A^\eps_t,f)$, whose continuity is clear from its definition.
 
Hence it remains to show \eqref{EK_1.16} and \eqref{EK_1.19} with \corr{$c=\frac{2}{\kappa_d}\langle af~|af\rangle$}.
We will start from proving \eqref{EK_1.16}.\\
We compute
        \begin{align*}
            \E\left[ \sup_{s\leq t} |(M^\eps_s,f)-(M^\eps_{s-},f)|^2\right] &= \E\left[ \sup_{n\leq N^\eps_t} |(\hat{h}^\eps_n,f)-(\tilde{h}^\eps_n,f)|^2\right] \\
            &\leq 4\E\left[ \sup_{n\leq N^\eps_t} |(\hat{h}^\eps_n,f)|^2\right] \\
            &= 4\int_0^\infty \P\Big(\sup_{n\leq N^\eps_t} |(\hat{h}^\eps_n,f)|^2 \geq x\Big)dx \\
            &\leq 4\eps + 4\int_\eps^\infty \P\Big(\exists n\leq N^\eps_t~:~ |(\hat{h}^\eps_n,f)|^4 \geq x^2\Big)dx\\
            &\leq 4\eps + 4\int_\eps^\infty \E[N^\eps_t]\P\big( |(\hat{h}^\eps_n,f)|^4 \geq x^2 \big)dx\\
            &\leq 4\eps + 4C\sup_{z\in D}|f(z)|^4\t t \eps^{2(2+d)}\int_\eps^\infty \frac{1}{x^2}dx,
        \end{align*}
        where we used MTD fourth moment condition in the last inequality. As $\t = C\eps^{-2-d}$ both summands converge to $0$ as $\eps \to 0$, as desired.

Let us now verify the convergence in \eqref{EK_1.19}. We will only consider the term $\frac{\t}{|D|} \int_0^t \int_D (\tilde{h}^\eps(s,x),f)^2 dx ds$, which we will denote by $(A^\eps_t,f)$. The term with the conditional expectation is deterministic thanks to MTD second moment and can be immediately treated to give $\frac{1}{\kappa_d}\langle af ~| af \rangle t$. 

Observe that \corr{$\lim_{\eps\mapsto 0}\E\left[ (A^\eps_t,f) \right] = \frac{1}{\kappa_d} \langle af~|~af\rangle t$}. So it is enough to show that
    \begin{align*}
        (A^\eps_t,f) - \E\left[(A^\eps_t,f) \right] = \dfrac{\t}{|D|} \sum_{k=1}^{N_t^\eps} (T_k^\eps - T_{k-1}^\eps) \int_D h^\eps_x - \E\left[ h^\eps_x \right] dx =: \sum_{k=1}^{N_t^\eps} X_{\eps,k},
    \end{align*}
    converges to 0 in probability as $\eps$ converges to 0, where $h^\eps_x = (\tilde{h}(T_{k-1}^\eps,x),f)^2$. We have, for each $\eps>0$ and $k\in \N_{>0}$ that $\E[X_{\eps,k}]=0$ and 
    \begin{align*}
        \E[X_{\eps,k}^2] = \left(\frac{\t}{|D|}\right)^2 \E[(T_k^\eps - T_{k-1}^\eps)^2] \int_D\int_D \text{Cov}\big(h^\eps_x - \E\left[ h^\eps_x\right]~;~h^\eps_y - \E\left[ h^\eps_y\right]\big)dxdy.
    \end{align*}
    From MTD second moment we can deduce that the double integral is zero for $|x-y|>2\eps$. Thus, applying Markov inequality and noticing that $T_k^\eps - T_{k-1}^\eps$ follows an exponential distribution of parameter $\t t$, we obtain
    $$ \E[X_{\eps,k}^2] \leq \frac{\t^2}{|D|^2} \frac{2}{(\t t)^2} (2\eps)^d \int_D \E[(h^\eps_x)^2]dx \leq C\eps^d \t^{-2}, $$
    for a certain constant $C>0$, where we used MTD fourth moment in the last inequality. Finally, we fix $a>0$ and we compute, using Chebychev inequality
    \begin{align*}
        \P\left( \left| \sum_{k=1}^{N_t^\eps} X_{\eps,k} \right|>a \right) &\leq \frac{\E \left| \sum_{k=1}^{N_t^\eps} X_{\eps,k} \right|^2 }{a^2} \leq \frac{(\t t)^2 \E[X_{\eps,k}^2]}{a^2} \leq C\frac{\eps^d}{a^2},
    \end{align*}
    for a certain constant $C>0$. The latter converges to 0 as $\eps$ converges to 0, such that \eqref{EK_1.19} holds.
\end{proof}


\subsection{The convergence of $\g_t$ as $\eps \to 0$}
\label{section_conv}

\subsubsection{Tightness of the family $(h^\eps)_{\eps>0}$}
In this section we will see how the previous results allow us to deduce that the family $(h^\eps)_{\eps>0}$ is tight in $D(\R^+;\Ss_D)$\footnote{$D(\R^+;\Ss_D)$ is the topological space of all c\`adl\`ag functions from $\R^+$ to $\Ss_D$}. 

\begin{prop}\label{prop:tighttriplet}
The triplet $((M_t^\eps)_{t \geq 0},(\g_t)_{t \geq 0}, (\int_0^t \int_D \Delta_\eps \g_s ds)_{t \geq 0})$ is tight in $D(\R^+, \Ss_D^3)$ and for each subsequential limit $(M_t, h^\infty_t, I_t)$ it holds that $M_t = h^\infty_t - I_t$ a.s..
\end{prop}

\begin{proof}
By \cite[Thm 4.1]{Mitoma}, we know that it suffices to prove that for every test function $f\in C^\infty_c(D)$ the family $(((M_t^\eps,f))_{t\geq 0}, (\g,f))_{t\geq 0}, (\int_0^t (\Delta_\eps \g_s,f)ds)_{t \geq 0})$ is tight in $D(\R^+;\R^3)$. 

So fix $f\in C^\infty_c(D)$. Denote also $I^\eps_t := \int_0^t \Delta_\eps h^\eps_s ds$ such that $h^\eps = M^\eps - I^\eps$ as distributions.\\

From the convergence in law of $((M^\eps,f))_{\eps > 0}$ and the fact that  $D(\R^+, \R)$ is a Polish space, Prokhorov's criterion tells us that $((M^\eps,f))_{\eps>0}$ is tight. Thus, as any joint law is tight as soon as each component is tight, and $(h^\eps,f) = (M^\eps,f) - (I^\eps,f)$ is tight as soon as $(M^\eps,f), (I^\eps,f)$  are, it suffices to prove that the family $((I^\eps,f))_{\eps>0}$ is tight in $D(\R^+;\R)$ (in fact it is even tight in the space of continuous functions). 

To show this, we will use the Kolmogorov-Chentsov theorem, in the form stated in \cite[Thm~III.8.8,~p.139]{EthierKurtz}. To be more precise, to obtain tightness it suffices to prove that:
\begin{itemize}
\item[\textit{i)}] For all $\alpha>0$, there exists a compact set $K\subset \R$ such that for all $\eps>0$ and $t\geq 0$,
$$ \P((I^\eps_t,f) \in K) \geq 1-\alpha. $$
\item[\textit{ii)}] There exists $C>0$ such that, for all $\eps>0$ and $0\leq \delta \leq t$ 
\begin{align}
    \label{tightness_obj}
    \E\left[ |(I^\eps_{t+\delta} - I^\eps_t,f)|^2\right] \leq C \delta^2.
\end{align} 
\end{itemize}
\paragraph{proof of i)}
The first point is immediately checked since $\g_t$ are following the same law as $h^D$ for every $\eps>0$ and $t\geq 0$. Thus, since $I^\eps_0$ is tight in $\R$, for any $\alpha>0$, there exists a compact set $K\subset \R$ such that for all $\eps>0$ and $t\geq 0$,
$$ \P((I^\eps_0,f) \in K) \geq 1-\alpha. $$
This $K$ works for all the $(I^\eps_t,f)$.
\paragraph{proof of ii)} From Jensen inequality and again the fact that each marginals of $(I^\eps,f)$ follows the same law, we obtain that
$$ \E\left[ \left( \int_0^t (\Delta_\eps h^\eps_s,f) ds \right)^2 \right] \leq t^2 \E[(\Delta_\eps h_D,f)^2].$$
Thus it suffices to prove that $\sup_{\eps>0}\E[((\Delta_\eps h_D,f))^2] <\infty$. But by Corollary \ref{cor:intbyph}, we have that $\E[(\Delta_\eps h_D,f)^2] = \E[(h_D,\Delta_\eps f)^2]$. With similar computations to what has been done in the proof of Proposition \ref{prop_delta_eps} we obtain that $\Delta_\eps f = \Delta f + R_f$ where the rest $R_f$ is given again by \eqref{reste_taylor}. Adding the integral and using MTD's second moment we get a bound uniform in $\eps$, which concludes the proof.
\end{proof}

\subsubsection{Conclusion of Theorem \ref{mainthmlemma}} 
\label{section_cv_ps}

Using Proposition \ref{prop:tighttriplet} we can now consider a subsequential limit of 
 $((M_t^\eps)_{t \geq 0}, (\g_t)_{t \geq 0}, (\int_0^t \tilde h^\eps_s ds)_{t \geq 0})$ that we denote by $(M_t, h^\infty_t, I_t)$ and for which it holds that $M_t = h^\infty_t - I_t$. 

 The only remaining step is basically identifying that $I_t$ is what we expect it to be:
\begin{lemma}\label{lemma:finalize}
Let $(M_t, h^\infty_t, I_t)$ be any subsequential limit of $((M_t^\eps)_{t \geq 0}, (\g_t)_{t \geq 0}, (\int_0^t \tilde h^\eps_s ds)_{t \geq 0})$, whose existence is given by Proposition \ref{prop:tighttriplet}. 

We have that $I_t(f) = \int_0^t (\Delta h^\infty_s,f)ds$. 
 \end{lemma}
Before proving the lemma, let us see how it implies Theorem \ref{mainthmlemma}

\begin{proof}[Proof of Theorem \ref{mainthmlemma}]
Proposition \ref{prop:tighttriplet}  implies that $((M_t^\eps)_{t \geq 0}, (\g_t)_{t \geq 0}, (\int_0^t \g_s ds)_{t \geq 0})$ is tight in the Skorokhod topology and if we denote by $(M_t, h^\infty_t, I_t)$ the limit, we have $M_t = h^\infty_t - I_t$.
 Further Lemma \ref{lemma:finalize} implies that for every $f \in C_c^\infty(D)  $ we have that $I_t = \int_0^t \Delta h^\infty_s ds$.

Now, by construction $h^\infty_t$ has the law of $h^D$ for any fixed $t$. Further, from Proposition \ref{prop:mCLT} we know that for each $f \in C_c^\infty(D)$, the process $(M_t,f) = (h^\infty_t,f) - \int_0^t (\Delta h^\infty_s,f) ds$ is a continuous centered Gaussian process with variance $\E[(M_t,f)^2]=\frac{2}{\kappa_d}\langle af~|~af\rangle t$. Thus we can apply Corollary \ref{cor:char} to conclude that $h^\infty_t$ is a solution of the $a\sqrt{2/\kappa_d}-$SHE. 

Finally, as every subsequential limit $h_t^\infty$ has the same law, we conclude the actual convergence in law in for the Skorokhod topology.
\end{proof}
 
Finally, it remains to prove Lemma \ref{lemma:finalize}.

\begin{proof}[Proof of Lemma \ref{lemma:finalize}]
By hypothesis, $\g \mapsto \gi$ in law, and by Skorokhod representation theorem we can suppose without loss of generality that the convergence is almost sure. By the expression \eqref{MP1_epsilon} it suffices then to prove that $\int_0^t\Delta_\eps \g_sds \mapsto \int_0^t\Delta \gi_sds$ almost surely for every $t\geq 0$.

For the integrand this follows from Corollary \ref{cor:deltah} by choosing $h_n = h_s^\eps$ and $h^D = h_s^\infty$. We can further compute\begin{align*}
    \E\left[\left| \int_0^t\Delta_\eps \g_sds - \int_0^t\Delta \gi_sds \right|\right] &\leq \int_0^t \E\left|\Delta_\eps \g_s - \Delta \gi_s\right|ds.
\end{align*}
By the MTD we have that the law of $\Delta_\eps \g_s - \Delta \gi_s$ does not depend on $s$, and we proved above the almost sure convergence. Thus to obtain the convergence in $L^1$ it is enough to show that the family $(\Delta h^\eps_0)_{\eps>0}$ is uniformly integrable. We have $\sup_{\eps>0}\E[(\Delta h^\eps_0)^2]=\sup_{\eps>0}\E[(\Delta h_D)^2]$ which we already showed is finite in the proof of \eqref{tightness_obj}. 

Hence we proved that $\int_0^t\Delta_\eps \g_sds \mapsto \int_0^t\Delta \gi_sds$ in $L^1$. Taking a further subsequential limit to have almost sure convergence then gives the desired conclusion.
\end{proof}


\section{The case of Fractional Gaussian free fields}
\label{section_fractional}

In this section we explain the modifications needed to generalize the proof above to fractional Gaussian free fields for $\alpha \in (0,1)$. In fact, only small modifications are needed and these are explained in Section \ref{sec:modif}. First we will however recall some concepts around the fractional Laplacian, including $\alpha$-harmonic functions / extensions and the fractional SHE.

\subsection{The fractional Laplacian, $\alpha$-harmonic functions / extensions and the fractional SHE}\label{section_frac_intro}

Recall that the Riesz fractional Laplacian given by
\begin{align*}
    (-\Delta)^\alpha u(z) = \lim_{\eps\mapsto 0} \int_{\R^d-B(z,\eps)} \frac{u(z)-u(y)}{|z-y|^{d+2\alpha}}dy,
\end{align*}
where we extended the values of $u$ outside of $D$ by $0$. Notice it is not given by taking the basis of eigenfunctions of $-\Delta$ and reweighting the coefficients, rather it is given by the restriction of the whole space Laplacian (defined by reweighing the Fourier coefficients in the Fourier transform), see \cite{FGF_review}.

For $\alpha \in (0,1)$ the fractional Green function $G_\B^\alpha$ in the ball can be explicitly computed \cite[Eq.~4.3]{FGF_review} or \cite[Section 1.1]{bogdan2009} and is given by:
$$G^\alpha_\B(x,y) = k_{\alpha,d}|x-y|^{2\alpha-d}\int_{1}^{\frac{||x|y-x/|x||}{|x-y|}}(v^2-1)^{\alpha -1}v^{1-d}dv,$$
with $k_{s,d} = \frac{\Gamma(1+d/2)}{d\pi^{d/2}4^{d-1}((\alpha-1)!)^2}$.
From here, a direct calculation shows that it gives rise to a compact integral operator in this regime. 

In general there are no explicit formulas for Green's functions arbitrary bounded domains $D$ with smooth boundary, however there is enough control on it to obtain again that the related integral operator is compact, see e.g. \cite[Thm ~1.1]{ChenSong}.

\subsubsection{The $\alpha$-mean property and $\alpha$-harmonic extension}

Recall the mean value property for harmonic functions: if $f$ is harmonic, then $f(z)$ can be obtained by testing $f$ against a smooth symmetric test function of unit mass around $z$. Such a nice general local formulation is no longer true for the fractional Laplacian in $\alpha \in (0,1)$, because it is non-local - the associated L\'evy process has jumps. Still a certain mean-value property persists. To state it we first introduce the $\alpha$-mean kernel:
\begin{defi}
    \label{defi_s_mean_kernel}
    For any ball $B(x,\eps)\subset D$, we define the $\alpha$-mean kernel as the function $A_x^\eps$ by
    $$A_x^\eps(y) ~ =  \left| 
\begin{array}{ccccc}
	  c(d,\alpha) \frac{\eps^{2\alpha}}{(|y-x|^2-\eps^2)^\alpha|y-x|^d} & \text{ if} & y\in \R^d-B(x,\eps) \\
	0 &\text{ if} & y\in \bar{B}(x,\eps)
\end{array},
\right. $$
where $c(d,\alpha)$ is a normalization constant such that $\int_{\R^d} A_x^\eps(y)dy =1$.
\end{defi}
\noindent For the exact value of $c(d,\alpha)$ and more details about the $\alpha$-mean kernel, one can look at \cite{Bucur, bucur_al}. Observe that the $\alpha$-mean kernel is not local nor smooth, thus a priori cannot be tested against a random distribution unless one assumes more regularity. For us the following scaling property will be important: 
\begin{lemma}
For any $x\in D$ and $\eps_1,\eps_2>0$,
\begin{align}
    \label{s_mean_kernel_scaling}
    A_x^{\eps_1}(2\eps_1 y) = \left(\dfrac{\eps_2}{\eps_1}\right)^d A_x^{\eps_2}(2\eps_2 y). 
\end{align}
\end{lemma}
The $\alpha$-mean value property then takes the following form and in fact characterises $\alpha$-harmonic functions, as shown in \cite[Thm 2.2]{Bucur}:
\begin{prop}[$\alpha$-mean value property]
    \label{lemme_mean_value_property_fractional}
    Let $0<\alpha<1$. A function $\varphi \in C^\infty_c(\R^d)$ is $\alpha$-harmonic in a ball $B\subset D$ if and only if for all $x\in B$ and $r>0$ such that $B(x,r)\subset B$ we have
    $$ \varphi(x) = \int_{\R^d} A_x^r(y)\varphi(y)dy. $$
\end{prop}
\noindent Notice that prolonging $\varphi$ by $0$ outside of $D$ becomes a necessity. The $\alpha$-mean kernel is closely related to the fractional Poisson kernel:
\begin{defi}
Let $\alpha\geq 0$. For each ball $B(x,\eps)\subset D$, the fractional Poisson kernel is defined by
\begin{align}
\label{def_poisson_kernel_fractional}
\begin{array}{ccccc}
	P^\alpha_{B(x,\eps)} & : &  B(x,\eps) \times \R^d-\bar{B}(x,\eps) & \longrightarrow & \R \\
	& & (z,y) & \longmapsto & c(d,\alpha)\left(\dfrac{\eps^2-|z-x|^2}{|x-y|^2-\eps^2}\right)^\alpha\dfrac{1}{|z-y|^d}
\end{array}.
\end{align}
Then, if $g$ is a continuous function on $C^0(\R^d)$ the fractional Poisson integral of $g$ is defined by, with $z\in B(x,\eps)$
\begin{align}
\label{def_poisson_int_fractional}
P^\alpha_{B(x,\eps)}[g](z) := \int_{\R^d-\bar{B}(x,\eps)} g(y)P^\alpha_{B(x,\eps)}(z,y)dy.
\end{align}
The function
$$u(z) ~ =  \left| 
\begin{array}{ccccc}
	P^\alpha_{B(x,\eps)}[g](z) & \text{ if} & z\in B(x,\eps) \\
	g(z) &\text{ if} & z\notin B(x,\eps)
\end{array},
\right. $$
is in $C^0(\R^d)$ and harmonic on $B(x,\eps)$.\\
\end{defi}

\subsubsection{The fractional stochastic heat equation}

\begin{defi}
The fractional stochastic heat equation (FSHE) on a bounded smooth $D$ is the SPDE,
\begin{align}
\label{FSHE}
\begin{cases}
& \partial_t u + \da u = a\dot{W} \\
&  u(0,x) = u_0(x), \qquad x\in D\\
& u(t,x) = 0, \qquad x\in \partial D,~ t\geq 0
\end{cases}
\end{align} 
where as before $\dot{W}$ is a space-time white noise; further $a(x)$ is a space-time constant and $u_0 \in \Ff_0$ is the initial condition. Again the zero boundary condition has to be interpreted in the sense of Definition \ref{defiZB}.
\end{defi}
For $\alpha \in (0,1)$, random-field solutions exist only in the low-dimensional/subcritical regimes; this distinction is not important here, because the argument only uses the distribution-valued formulation. We will only need the following two facts that are direct to establish similarly to the SHE case:
\begin{itemize}
    \item The FSHE can be rewritten as $\sum_{n=0}^\infty A_n(t)f_n$ where the $A_n(t)$ are solutions to the Ornstein-Uhlenbeck processes with parameter $\mu_n^\alpha$ 
$$ dX_t = -\mu_n^\alpha X_t dt + dB_t^n, $$
and the $B_t^n$ are Brownian motions of covariation given by $\langle B_t^n,B_t^m\rangle = a^2t\langle f_n~|~ f_m\rangle$. Here $(f_n)_{n \geq 1}$ is the basis of eigenfunctions of $(-\Delta)^\alpha$ guaranteed by the fact that the Green's operator is compact.\\
    \item Secondly, we have that the $\fgf$ is the unique stationary solution of the FSHE. That can be seen from the previous point or by generalizing the arguments in \cite{LototskyShah}
\end{itemize}

\subsection{Modifications of the proof in the fractional case}\label{sec:modif}

Here we explain the small modifications needed to adapt the proof from the GFF case to the FGF case. 

\subsubsection{Preliminary section and the definition of dynamics}
The section \ref{sec:prelimGFF} works pretty much the same, only when obtaining the estimates on the two-point function one cannot use smooth mollifiers and instead has to use the $\alpha-$mean kernel from just above. 
More precisely, the analogues of Lemmas \ref{lemorth}, \ref{remarkpropmarkov1} work as stated. In Lemmas \ref{remarkpropmarkov2}, \ref{remarkpropmarkov3}, \ref{lemmekbound} the function $s_d$ has to be replaced by $s_{\alpha,d}(x) := |x|^{2\alpha -d}$ and the proofs need small modifications, for the convenience of the reader we give the proof of the analogue of Lemma \ref{remarkpropmarkov2} in Appendix \ref{appendix_lemmafrac}. The analogue of Lemma \ref{lemme_k_cov_kernel} works well again and Lemma \ref{lem:ctyk2} and Proposition \ref{prop:Green} are not needed - the first one because of our extra assumption on the martingale-type decomposition the second because of the simplifying assumption that $a(x)$ is constant. 

The definition of the dynamics and conclusion of the theorem works the same, with $\t = C\eps^{-2-d}$ replaced by $\t = C\eps^{-2\alpha -d}$ and the stochastic heat equation by the fractional stochastic heat equation. Recall that here we are working under the assumption $a(x)$ constant, so we don't need to evoke any further results on the covariance kernel to conclude.

\subsubsection{The approximation of $(-\Delta)^\alpha$}
\label{section_adaptation_frac_laplacian}
In this section some actual modifications are needed. Indeed, whereas the analogues of Lemma \ref{lemma_Gamma_approx_laplacian} and Proposition \ref{prop:harmonic} work quite similarly and are even simpler, due to the extra assumption in A' connecting directly the $\alpha-$harmonic part with the $\alpha-$harmonic extension. For example continuity of $k_D$ at the end of Proof of Lemma \ref{lemma_Gamma_approx_laplacian} can be swapped because of this. However, the proof of Proposition \ref{prop_delta_eps}, i.e. the approximation of the Laplacian has to reworked, thus let us state and prove it.

First, the approximation of $(-\Delta)^\alpha$:
\begin{defi}[Approximation of $\da$]
Let $\t = C\eps^{-2\alpha -d}$. For any fixed $\eps>0$ we define the operator $\dea ~:~C^0(\R^d) \longrightarrow C^0(\R^d)$ by 
\begin{align}
\label{Delta_epsilon_frac}
\dea g(z) := \dfrac{\t}{V_d}\int_{B(z,\eps)} \left(g(z)-P^\alpha_{B(x,\eps)}[g](z)\right)dx.
\end{align}
Equivalently for any test function $f \in C_c^\infty(\R^d)$ we set
\begin{equation}
\label{Delta_epsilon2_frac}
(\dea g, f) := \dfrac{\t}{|V_d|}\int_D\left(\int_{B(x,\eps)} \left(g(z)f(z)-P^\alpha_{B(x,\eps)}[g](z)f(z)\right)dz\right)dx.
\end{equation}
\end{defi}

We then have
\begin{prop}[Convergence of $\dea$ to $\da$]
\label{prop_delta_eps_frac}
There exists an explicit constant $C>0$ such that if $\t:=C\eps^{-(2\alpha+d)}$ then for any family $(g)_{g\in\Gg}$ in $C^\infty_c(\R^d)$ bounded in $\Ss_D$, we have that in $\Ss_D$
$$ \dea g \longmapsto \da g, \qquad \text{as }\eps\mapsto 0.$$
Moreover, this convergence is pointwise uniform, i.e. for any fixed $f\in C^\infty_c(D)$ we have
$$\sup_{g\in \Gg} |(\dea g - \da g,f)| \mapsto 0,\qquad \text{as }\eps\mapsto 0.$$ 
\end{prop}
We will give a full proof of Proposition \ref{prop_delta_eps_frac} as it differs considerably from the proof of Proposition \ref{prop_delta_eps}. The main reason is that $(-\Delta)^\alpha g(z)$ is not a term of the Taylor's expansion of $g$ and thus the spherical symmetry used in the end of the proof does not hold anymore. We circumvented this by using the fact that the operator $(-\Delta)^\alpha g(z)$ as expressed in \eqref{laplacianfractionalintegral} behaves well in conjunction with the fractional Poisson kernel.
\begin{proof}
We compute, with \eqref{Delta_epsilon_frac},
\begin{align*}
    \dea g(z)&= \frac{\t}{V_d} \int_{B(z,\eps)} \int_{\R^d-B(x,\eps)} (g(z)-g(y))P^\alpha_{B(x,\eps)}(z,y)dydx\\
    &= c(d,\alpha)\frac{\t}{V_d} \int_{\R^d} \frac{g(z)-g(y)}{|z-y|^d} \int_{(\R^d-B(y,\eps))\cap B(z,\eps)} \left( \frac{\eps^2-|z-x|^2}{|x-y|^2-\eps^2} \right)^\alpha dxdy.
\end{align*}
Since when $|y-z|>2\eps$ we have $B(z,\eps) \subset \R^d-B(y,\eps)$, we can write
\begin{align*}
    \dea g(z)&= c(d,\alpha)\frac{\t}{V_d} \int_{\R^d-B(z,2\eps)} \frac{g(z)-g(y)}{|z-y|^d} \int_{B(z,\eps)} \left( \frac{\eps^2-|z-x|^2}{|x-y|^2-\eps^2} \right)^\alpha dxdy \\
    &+ c(d,\alpha)\frac{\t}{V_d} \int_{B(z,2\eps)} \frac{g(z)-g(y)}{|z-y|^d} \int_{(\R^d-B(y,\eps))\cap B(z,\eps)} \left( \frac{\eps^2-|z-x|^2}{|x-y|^2-\eps^2} \right)^\alpha dxdy
\end{align*}
For the first term of the sum, by Lebesgue's dominated convergence theorem we have, noting $w:=\frac{z-y}{\eps}$,
$$|w|^{2\alpha}  \int_{B(0,1)} \left( \dfrac{1 - |x|^2}{|w-x|^2-1}\right)^\alpha dx \longrightarrow \int_{B(0,1)} ( 1 - |x|^2)^\alpha dx,$$
as $\eps\mapsto 0$, which since by linear substitution $\int_{B(z,\epsilon)} \left( \frac{\epsilon^2 - |z-x|^2}{|x-y|^2-\epsilon^2}\right)^\alpha dx =\epsilon^d \int_{B(0,1)} \left( \frac{1 - |x|^2}{|w-x|^2-1}\right)^\alpha dx$, implies that
\begin{align*}
    & \lim_{\eps\mapsto 0} c(d,\alpha)\frac{\t}{V_d} \int_{\R^d-B(z,2\eps)} \frac{g(z)-g(y)}{|z-y|^d} \int_{B(z,\eps)} \left( \frac{\eps^2-|z-x|^2}{|x-y|^2-\eps^2} \right)^\alpha dxdy \\
    =& \lim_{\eps\mapsto 0} c(d,\alpha)\frac{\t}{V_d} \int_{\R^d-B(z,2\eps)} \frac{g(z)-g(y)}{|z-y|^d} \int_{B(0,1)} ( 1 - |x|^2)^\alpha dx \frac{\eps^{d+2\alpha}}{|z-y|^{2\alpha}} dy \\
    =& c(d,\alpha) \int_{B(0,1)} ( 1 - |x|^2)^\alpha dx~ \frac{\t \eps^{d+2\alpha}}{V_d} \lim_{\eps\mapsto 0} ~\int_{\R^d-B(z,2\eps)} \frac{g(z)-g(y)}{|z-y|^{2\alpha+d}} dy \\
    =& c(d,\alpha) \int_{B(0,1)} ( 1 - |x|^2)^\alpha dx~ \frac{\t \eps^{d+2\alpha}}{V_d} \da g(z),
\end{align*}
where the limit is uniform on every compact subset of $D$. We see that  we can choose 
$$\t := V_d\left(c(d,\alpha) \eps^{d+2\alpha} \int_{B(0,1)} ( 1 - |x|^2)^\alpha dx \right)^{-1}.$$
\corr{
For the second term of the sum, which we want to show vanishes as $\eps\to0$, set $\omega=(z-y)/\eps$ and define
\[
 c_0(\omega):=\int_{(\R^d-B(\omega,1))\cap B(0,1)} \left(\frac{1-|u|^2}{|u-\omega|^2-1}\right)^\alpha du .
\]
By rotational invariance $c_0(\omega)$ only depends on $|\omega|$. The affine substitution $x=z+\eps u$ gives
\[
\int_{(\R^d-B(y,\eps))\cap B(z,\eps)} \left( \frac{\eps^2-|z-x|^2}{|x-y|^2-\eps^2} \right)^\alpha dx
=\eps^d c_0\!\left(\frac{z-y}{\eps}\right).
\]
Thus the near-field term is
\[
T_\eps(z)=C_{\alpha,d}\frac{\t\eps^d}{V_d}
\int_{B(z,2\eps)}\frac{g(z)-g(y)}{|z-y|^d}
 c_0\!\left(\frac{z-y}{\eps}\right)dy .
\]
Changing variables $y=z+\eps r\theta$, with $0<r<2$ and $\theta\in\mathbb S$, yields
\[
T_\eps(z)=C_{\alpha,d}\frac{\t\eps^d}{V_d}
\int_0^2\frac{c_0(r)}{r}\int_{\mathbb S}\big(g(z)-g(z+\eps r\theta)\big)d\theta\,dr .
\]
The first-order Taylor term cancels by symmetry and
\[
\int_{\mathbb S}\big(g(z)-g(z+\eps r\theta)\big)d\theta
=-\frac{|\mathbb S|}{2d}\eps^2r^2\Delta g(z)+O(\eps^3r^3),
\]
uniformly for $g$ in bounded smooth families. By Lemma \ref{lemme_control_c}, $c_0(r)\leq Cr$ near $0$, and $c_0$ is locally bounded on $(0,2)$; hence $\int_0^2 r c_0(r)dr<\infty$. Since $\t\eps^d\eps^2=O(\eps^{2-2\alpha})$ and $\alpha<1$, we get $T_\eps(z)\to0$ uniformly on compact subsets.}

\end{proof}

\begin{lemm}
\label{lemme_control_c}
Define the radial function $c_0(r)$ as in the proof of Proposition \ref{prop_delta_eps_frac}. There exists a constant $C>0$ such that for any $r<1$, $c_0(r) \leq Cr$.
\end{lemm}
\begin{proof}
First, note that for any $x\in (\R^d-B(\omega,1))\cap B(0,1)$ where $|\omega|=r$, we have $1-|x|^2 \leq r(2-r)$.\\
Moreover, we have the inclusion 
\begin{align}
    \label{eq_figure}
    (\R^d-B(\omega,1))\cap B(0,1) \subset \bigcup_{t\in[0,r]} S^{1/2}(\omega,1+t),
\end{align}
where $S^{1/2}(\omega,1+t)$ is the half of the sphere $S(\omega,1+t)$ the closest to the origin, as shown in Figure \ref{fig_preuve}. Thus, we can compute
\begin{align*}
    c_0(r) &\leq \int_0^r \int_{S^{1/2}(\omega,1+t)} \left(\dfrac{1-|x|^2}{|x-\omega|^2-1}\right)^\alpha dxdt\\
    &\leq \int_0^r \int_{S^{1/2}(\omega,1+t)} \left(\dfrac{r(2-r)}{t(2+t)}\right)^\alpha dxdt\\
    &=Cr^\alpha \int_0^r\frac{1}{t^\alpha}dt = Cr^\alpha r^{1-\alpha}=Cr,
\end{align*}
for a certain constant $C>0$.
\end{proof}

\begin{figure}
\centering
    \includegraphics[width=0.5\linewidth]{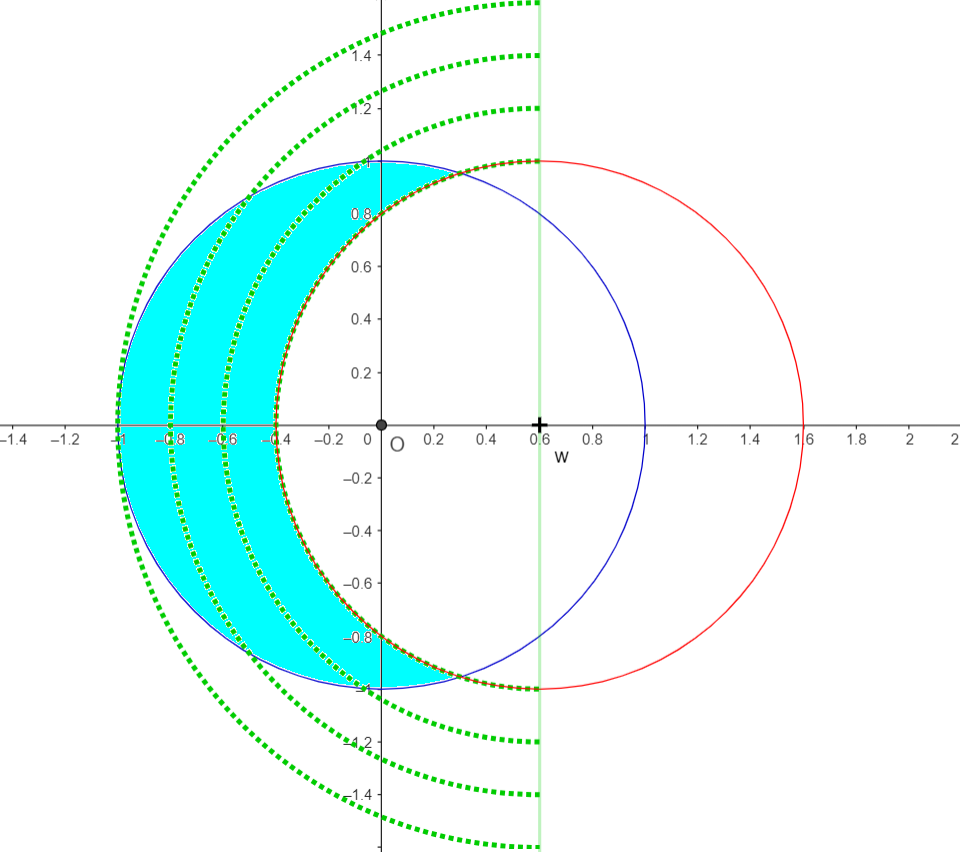}
    \caption{Figure of the situation for $d=2$ and $r=0.6$.\\
    The left hand side set of \eqref{eq_figure} is in blue, the sets $S^{1/2}(\omega,1+t)$ for $t\in \{0;0.2;0.4;0.6\}$ are shown in green}
    \label{fig_preuve}
\end{figure}

\subsubsection{Rest of the proof}
The rest of the proof of Theorem \ref{mainthmfrac} can be directly adapted from its counterparts from the proof of Theorem \ref{mainthm}, with only a few minor changes detailed at the beginning of this section.

\appendix
\section{Proof of Proposition \ref{propSHEOU}}
\label{appendix_walsh}
In this subsection we will prove Proposition \ref{propSHEOU} by following a strategy similar to the one in \cite{Walsh}.

As in \cite[p.~343]{Walsh}, using \eqref{SHEgreen} we obtain that the solution $u_t$ of the SHE as defined in \eqref{SHEsol} can be written, for every $f\in C^\infty(\R^d)$,
$$u_t(f) = \sum_{n=0}^\infty \langle f~|~f_n\rangle \int_0^t\int_D a(y)e^{-\mu_n(t-s)}f_n(y)W(dy,ds).$$
Thus, noting $A_n(t):=\int_0^t\int_D a(y)e^{-\mu_n(t-s)}f_n(y)W(dy,ds)$, this justifies the notation with the infinite sum
$$u_t := \sum_{n=0}^\infty A_n(t) f_n.$$
We will now prove that the random variables $A_n(t)$ are solutions to the Itô SDEs \eqref{propSHEOUeq}, and that the $B_t^n$ are Brownian motions of covariation given by $\text{Cor}(B_t^n,B_s^m) = t \langle af_n~|~af_m\rangle$.\\
~\\
Noting $B_t^n := \int_0^t\int_D a(y)f_n(y)W(dy,ds)$, we can immediatly see with Itô's formula that the $A_n(t)$ checks \eqref{propSHEOUeq}. Moreover, a common property of Walsh's integral is that each $B_t^n$ are martingales whose quadratic covariation is given by
$$ \text{Cor}(B_t^n,B_s^m) = \int_0^t\int_D a(x)f_n(x)a(x)f_m(x)dx\,ds = t \langle af_n~|~af_m\rangle .$$
To conclude, one should still prove that this characterizes uniquely $u_t$. To do so, since the Ornstein-Uhlenbeck SDEs \eqref{propSHEOUeq} admit a unique weak solution, it is enough to prove that the joint law of the process $(B_t^n~:~t\geq 0)_{n\geq 0}$ is uniquely determined by the condition $\text{Cor}(B_t^n,B_t^m)=t \langle  af_n~|~af_m\rangle$.\\
We choose $k\geq 1$, a finite subset $(n_1,\dots,n_k)\subset \N$ and $c=(c_1,\dots,c_k)\in\R^k$. We pose $m_t := \sum_{i=1}^k c_i B_t^{n_i}$ a martingale such that, noting $\langle\cdot~|~\cdot\rangle_t$ the quadratic variation,
\begin{align*}
\langle m\rangle_t = \sum_{i,j=1}^k c_ic_j \langle B^{n_i}~|~B^{n_j}\rangle_t = t\sum_{i,j=1}^k c_ic_j \langle a f_{n_i}~|~ a f_{n_j}\rangle = t (c^\dagger \cdot A \cdot c),
\end{align*}
where $A:=(\langle a f_{n_i}~|~ a f_{n_j}\rangle)_{1\leq i,j\leq k}$ and $v^\dagger$ is the transposed vector of a vector $v\in\R^d$.
Thus the exponential martingale associated to $m_t$ is 
$$\Ee(im)_t = \exp\left(im_t+\frac{t}{2}(c^\dagger \cdot A \cdot c)\right),$$
and is also a martingale. Thus, for $0\leq s\leq t$,
$$\E\left[\exp\left(ic^\dagger\cdot (m_t-m_s)\right)~|~\Ff_s\right] = \exp\left(-\dfrac{t-s}{2}(c^\dagger \cdot A \cdot c)\right).$$
We recognize the characteristic function of a multidimensional normal distribution such that the law of $m_t-m_s$ conditionnally to $\Ff_s$ is $\Nn(0,(t-s)A)$. Thus, $m_t$ follows the law of a continuous L\'evy process of $\Nn(0,(t-s)A)$ increments. This characterizes uniquely the law of $m_t$, and furthermore the law of the process $(B_t^n~:~t\geq 0)_{n\geq 0}$.

\section{Lemma for identifying martingales}
\begin{lemm}
\label{lemme_martingale}
Let $(\Omega,\Ff,(\Ff_t)_{t\geq 0},\P)$ be a filtered probability space.\\
If $(X_t)_t$ is a $\Ff_t$-adapted process taking values in a Banach space $(B,|\cdot|)$ such that there exist $\Ff_t$-adapted processes $(A_t)_{t\geq 0}$ and $(o_s(t))_{s,t\geq 0}$ verifying for every $h:=t-s$,
$$ \E[X_t~|~\Ff_s]=X_s+h A_s + o_s(h), $$
where
\begin{itemize}
\item[(i)] the random variables $o_s$ check $\E|o_s(h)|\leq C o(h)$, where $C$ is a constant,
\item[(ii)] the function $t \mapsto A_t$ is almost surely cadlag and Bochner-integrable,
\item[(iii)] for all $T>0$, we have $\E\sup_{t\leq T} |A_t|<\infty$.
\end{itemize} 
Then the process
$$Y_t := X_t - \int_0^t A_s ds,$$
is a $\Ff_t$-martingale.
\end{lemm}

\begin{proof}
Let $t>s\geq 0$, and prove that 
\begin{align}
\label{lemme_martingale_proof}
\E\left| ~ \E\left[ X_t-X_s - \int_s^t A_r dr ~\Big|~\Ff_s \right] ~ \right| =0. 
\end{align} 
We compute, with $\cup_{i=1}^n [s_i,t_i)$ an uniform partition of $[s,t)$ (\textit{ie} such that $t_i-s_i=h$ and $nh=t-s$)
\begin{align*}
\E\left| ~ \E\left[ X_t-X_s - \int_s^t A_r + dr ~\Big|~\Ff_s \right] ~ \right| &= \E\left| ~ \E \left[ \sum_{i=1}^n X_{t_i}-X_{s_i} - \int_s^t A_r dr ~\Big|~\Ff_s \right] ~ \right| \\
&= \E\left| ~ \E \left[ \sum_{i=1}^n hA_{s_i} + o_{s_i}(h) - \int_s^t A_r dr ~\Big|~\Ff_s \right] ~ \right| \\
&\leq \E\left| \sum_{i=1}^n hA_{s_i} - \int_s^t A_r dr \right| + \sum_{i=1}^n\E|o_{s_i}(h)|
\end{align*}
By the assumption \textit{(i)}, we have that $\sum_{i=1}^n\E|o_{s_i}(h)| \mapsto 0$ for $n\mapsto \infty$.\\
By the assumption \textit{(ii)} we have that $\sum_{i=1}^n hA_{s_i} \mapsto \int_s^t A_r dr$ almost surely for $n\mapsto\infty$, and with assumption \textit{(iii)} we can use the Bochner dominated convergence theorem to prove that the convergence is also in $L^1(\Omega)$.\\
Thus, \eqref{lemme_martingale_proof} indeed holds and $Y_t$ is a $\Ff_t$-martingale.
\end{proof}

\section{Lemma \ref{remarkpropmarkov3}}\label{appendix_lemma4}

\begin{proof}[Proof of Lemma \ref{remarkpropmarkov3}]
From Lemma \ref{remarkpropmarkov1}
\begin{align*}
\E\left[\varphi_D^{B(x,\eps')}(x)^4\right] &= \E\left[\varphi_D^{B(x,\eps)}(x)^4\right] + \E\left[\left((\tilde{h}^{B(x,\eps)},\eta_x^{\eps'})-(\tilde{h}^{B(x,\eps')},\eta_x^{\eps'})\right)^4\right] \\
&+ 6\E\left[\varphi_D^{B(x,\eps)}(x)^2\left((\tilde{h}^{B(x,\eps)},\eta_x^{\eps'})-(\tilde{h}^{B(x,\eps')},\eta_x^{\eps'})\right)^2\right].
\end{align*}
and the first point follows.
Now, define $2^{-n} \leq \eps < 2^{-(n-1)}$ and $n_0:=\inf\{m\leq n ~:~2^{-m}<\delta\}$. Then $\E\left[\varphi_D^{B(x,\eps)}(x)^2\right] \leq \E\left[\varphi_D^{B(x,2^{-n})}(x)^2\right] $.
Using the above identity several times we obtain
\begin{align*} 
& \E\left[\varphi_D^{B(x,2^{-n})}(x)^4\right] = \E\left[\varphi_D^{B(x,2^{-n_0})}(x)^4\right] + \sum_{m=n_0}^{n-1}\E\left[\varphi_{B(x,2^{-m})}^{B(x,2^{-(m+1)})}(x)^4\right]\\
&+ \sum_{m=n_0}^{n-1} \E\left[\varphi_D^{B(x,2^{-n_0})}(x)^2\varphi_{B(x,2^{-m})}^{B(x,2^{-(m+1)})}(x)^2\right] + 2\sum_{m>m'=n_0}^{n-1} \E\left[\varphi_{B(x,2^{-m})}^{B(x,2^{-(m+1)})}(x)^2\varphi_{B(x,2^{-m'})}^{B(x,2^{-(m'+1)})}(x)^2\right].
\end{align*}
We can bound this by
$$\E\left[ (h^D,\eta_x^{2^{-n_0}})^4 \right] + \sum_{m=n_0+1}^{n-1}\E\left[ (\tilde{h}^{B(x,2^{-(m-1)})},\eta_x^{2^{-m}})^4\right] + \E\left[ (h^D,\eta_x^{2^{-n_0}})^2\right]\sum_{m=n_0+1}^{n-1}\E\left[ (\tilde{h}^{B(x,2^{-(m-1)})},\eta_x^{2^{-m}})^4\right].$$
The uniform bound on the 4th moments (point B in the definition) bounds the first term by a constant that depends on $\delta$. Further $\sup_{x \in D}\eta_x^{2^{-m}} \leq 2^{md}$ and thus the MTD condition bounds each of the other terms by $2^{4md} 2^{-(m-1)2(d+2)}=2^{2(d+2)}2^{2m(d-2)}$. Finally the third term can be handled as in the proof of Lemma \ref{remarkpropmarkov2} with a constant depending on $\delta$. Summing now together, we obtain the relevant bounds.\\
Using H\"older's inequality, one can deduce the last identity.\end{proof}

\section{The analogue of Lemma \ref{remarkpropmarkov2} in the fractional case}\label{appendix_lemmafrac}

\begin{proof}
From the analog of Lemma \ref{remarkpropmarkov1} and Assumption A' we can write
\item[~]
$$ \E\left[\varphi_D^{B(x,\eps')}(x)^2\right] = \E\left[\varphi_D^{B(x,\eps)}(x)^2\right] + \E\left[\left((\tilde{h}^{B(x,\eps)},A_x^{\eps'})-(\tilde{h}^{B(x,\eps')},A_x^{\eps'})\right)^2\right]$$
and the first point follows.\\
Now, define $2^{-n} \leq \eps < 2^{-(n-1)}$ and $n_0:=\inf\{m\leq n ~:~2^{-m}<\delta\}$. Then $\E\left[\varphi_D^{B(x,\eps)}(x)^2\right] \leq \E\left[\varphi_D^{B(x,2^{-n})}(x)^2\right] $.
Using the above identity several times we obtain
$$ \E\left[\varphi_D^{B(x,2^{-n})}(x)^2\right] = \E\left[\varphi_D^{B(x,2^{-n_0})}(x)^2\right] + \sum_{m=n_0+1}^{n-1}\E\left[\left( (\tilde{h}^{B(x,2^{-(m-1)})},A_x^{2^{-m}}) - (\tilde{h}^{B(x,2^{-m})},A_x^{2^{-m}}) \right)^2\right].$$
We can bound this by
$$\E\left[ (h^D,A_x^{2^{-n_0}})^2 \right] + \sum_{m=n_0+1}^{n-1}\E\left[ (\tilde{h}^{B(x,2^{-(m-1)})},A_x^{2^{-m}})^2\right].$$
With FMTD's second moment and \eqref{s_mean_kernel_scaling}, we have $\E\left[(\tilde{h}^{B(0,2\eps)},A_x^\eps)^2\right] = \eps^{2\alpha+d}\E\left[(h^D,A_x^\eps(2\eps\cdot))^2\right] = \eps^{2\alpha-d}\E\left[(h^D,A_x^{1/2})^2\right]$, such that we get
\begin{align*}
    \E\left[\varphi_D^{B(x,\eps)}(x)^2\right] \leq C\left(1+\sum_{m=n_0+1}^{n-1} 2^{m(d-2\alpha)}\right),
\end{align*}
where the constant $C>0$ depends on $C_a$ and $n_0$ and thus on $\delta$. Computing the sum, we obtain 
\begin{align*}
    \E\left[\varphi_D^{B(x,\eps)}(x)^2\right] \leq C(\delta) \left(1+\eps^{2\alpha-d}\right)
\end{align*}
when $2\alpha<d$. 
\end{proof}

\bibliographystyle{alpha}
\bibliography{biblio}

\end{document}